\documentclass[english,12pt]{amsart}
\usepackage{amssymb}
\usepackage{amsfonts}
\usepackage{amscd}
\usepackage{longtable}
\usepackage{pb-diagram}
\usepackage[all]{xy}
\usepackage{color}
\usepackage{enumerate}
\usepackage[utf8]{inputenc}
\usepackage{mathtools}
\usepackage{todonotes}

\usepackage[T1]{fontenc}
\RequirePackage[colorlinks=true, breaklinks=true, urlcolor= black, linkcolor= black, citecolor= black, bookmarksopen=true,linktocpage=true,plainpages=false,pdfpagelabels]{hyperref}

\allowdisplaybreaks

\CompileMatrices

\swapnumbers

\def\rad{\operatorname{rad}}

\def\radop{\rad_{\mathrm{op}}}

\def\ac{{\rm ac}}

\def\val{{\mathrm{val}}}

\def\11{{\mathbf 1}}

\def\NN{{\mathbb N}}

\def\QQ{{\mathbb Q}}
\def\RR{{\mathbb R}}

\def\ZZ{{\mathbb Z}}

\def\cA{{\mathcal A}}
\def\cB{{\mathcal B}}

\def\cL{{\mathcal L}}
\def\cM{{\mathcal M}}

\def\cO{{\mathcal O}}

\def\cT{{\mathcal T}}

\mathchardef\alphag="7C0B \mathchardef\betag="7C0C
\mathchardef\gammag="7C0D \mathchardef\deltag="7C0E
\mathchardef\varepsilong="7C22 \mathchardef\varphig="7C27
\mathchardef\psig="7C20 \mathchardef\zetag="7C10
\mathchardef\epsilong="7C0F \mathchardef\rhog="7C1A
\mathchardef\taug="7C1C \mathchardef\upsilong="7C1D
\mathchardef\iotag="7C13 \mathchardef\thetag="7C12
\mathchardef\pig="7C19 \mathchardef\sigmag="7C1B
\mathchardef\etag="7C11 \mathchardef\omegag="7C21
\mathchardef\kappag="7C14 \mathchardef\lambdag="7C15
\mathchardef\mug="7C16 \mathchardef\xig="7C18
\mathchardef\chig="7C1F \mathchardef\nug="7C17
\mathchardef\varthetag="7C23 \mathchardef\varpig="7C24
\mathchardef\varrhog="7C25 \mathchardef\varsigmag="7C26
\mathchardef\Omegag="7C0A \mathchardef\Thetag="7C02
\mathchardef\Sigmag="7C06 \mathchardef\Deltag="7C01
\mathchardef\Phig="7C08 \mathchardef\Gammag="7C00
\mathchardef\Psig="7C09 \mathchardef\Lambdag="7C03
\mathchardef\Xig="7C04 \mathchardef\Pig="7C05
\mathchardef\Upsilong="7C07

\newtheorem{theorem}[subsubsection]{Theorem}
\newtheorem{thm}[subsubsection]{Theorem}
\newtheorem{lem}[subsubsection]{Lemma}

\newtheorem{prop}[subsubsection]{Proposition}

{Addendum}

\theoremstyle{definition}
\newtheorem{defn}[subsubsection]{Definition}

\newtheorem{example}[subsubsection]{Example}

\newtheorem{def-prop}[subsubsection]{Proposition-Definition}
\newtheorem{def-theorem}[subsubsection]{Theorem-Definition}
\newtheorem{def-lem}[subsubsection]{Lemma-Definition}

\theoremstyle{remark}
\newtheorem{remark}[subsubsection]{Remark}

\theoremstyle{plain}

\numberwithin{equation}{subsection}

\def\boxit#1#2{\setbox1=\hbox{\kern#1{#2}\kern#1}%
\dimen1=\ht1 \advance\dimen1 by #1 \dimen2=\dp1 \advance\dimen2 by
#1
\setbox1=\hbox{\vrule height\dimen1 depth\dimen2\box1\vrule}%
\setbox1=\vbox{\hrule\box1\hrule}%
\advance\dimen1 by .4pt \ht1=\dimen1 \advance\dimen2 by .4pt
\dp1=\dimen2 \box1\relax}

\renewcommand{\theequation}{\thesubsection.\arabic{equation}}

\mathchardef\alphag="7C0B \mathchardef\betag="7C0C
\mathchardef\gammag="7C0D \mathchardef\deltag="7C0E
\mathchardef\varepsilong="7C22 \mathchardef\varphig="7C27
\mathchardef\psig="7C20 \mathchardef\zetag="7C10
\mathchardef\epsilong="7C0F \mathchardef\rhog="7C1A
\mathchardef\taug="7C1C \mathchardef\upsilong="7C1D
\mathchardef\iotag="7C13 \mathchardef\thetag="7C12
\mathchardef\pig="7C19 \mathchardef\sigmag="7C1B
\mathchardef\etag="7C11 \mathchardef\omegag="7C21
\mathchardef\kappag="7C14 \mathchardef\lambdag="7C15
\mathchardef\mug="7C16 \mathchardef\xig="7C18
\mathchardef\chig="7C1F \mathchardef\nug="7C17
\mathchardef\varthetag="7C23 \mathchardef\varpig="7C24
\mathchardef\varrhog="7C25 \mathchardef\varsigmag="7C26
\mathchardef\Omegag="7C0A \mathchardef\Thetag="7C02
\mathchardef\Sigmag="7C06 \mathchardef\Deltag="7C01
\mathchardef\Phig="7C08 \mathchardef\Gammag="7C00
\mathchardef\Psig="7C09 \mathchardef\Lambdag="7C03
\mathchardef\Xig="7C04 \mathchardef\Pig="7C05
\mathchardef\Upsilong="7C07

\newcommand{\ovf}{\mathrm{ovf}}

\newcommand{\RV}{\mathrm{RV}}
\newcommand{\VF}{\mathrm{VF}}
\newcommand{\VG}{\mathrm{VG}}
\newcommand{\RF}{\mathrm{RF}}
\newcommand{\rv}{\operatorname{rv}}

\newcommand{\res}{\operatorname{res}}
\newcommand{\Th}{\operatorname{Th}}

\newcommand{\alg}{{\mathrm{alg}}}

\newcommand{\rcl}{\mathrm{rcl}}
\newcommand{\oag}{\mathrm{oag}}
\newcommand{\an}{\mathrm{an}}

\newcommand{\ifin}[1]{{#1}^{\circ \circ}} 	

\newcommand{\Kinf}{K^{\circ \circ}} 
\newcommand{\Kfin}{K^{\circ}}       

\newcommand{\Lor}{\cL_{\mathrm{or}}}

\newcommand{\abs}[1]{\lvert#1\rvert}
\DeclarePairedDelimiter{\norm}{\lVert}{\rVert}
\DeclarePairedDelimiter{\bignorm}{\big\lVert}{\big\rVert}

\definecolor{immi}{rgb}{0,.6,.1}

\newbox\removebox
\newcommand\remove[1]{%
\setbox\removebox=\ifmmode\hbox{$#1$}\else\hbox{#1}\fi%
\leavevmode
\rlap{\textcolor{blue}{\vrule height0.8ex depth-0.6ex width\wd\removebox}}%
\box\removebox
}
\long\def\bigremove#1{%
\par\setbox\removebox=\vbox{#1}%
\vbox{%
\vbox to0pt{\hbox{\tikz\draw[color=blue,thick] (0,0) -- (\wd\removebox,-\ht\removebox)  (\wd\removebox,0) -- (0,-\ht\removebox);}}
\box\removebox
}
}

\newcommand\qftp{\mathrm{qftp}}

\newcommand\dcl{\mathrm{dcl}}

\definecolor{orange}{rgb}{1,0.5,0}

\newcommand{\private}[1]{\leavevmode{\scriptsize\color{blue}\marginpar{{\scriptsize Private comment}}#1\par}}
\renewcommand{\private}[1]{}

\thanks{}

\title{Almost real closed fields with real analytic structure}

\author{Kien Huu Nguyen, Mathias Stout, and Floris Vermeulen}

\subjclass[2020]{Primary 03C64, 32P05, 32B05; Secondary 03C60,  14P15}

\keywords{analytic structure, subanalytic sets, quantifier elimination, Weierstrass	systems, henselian valued fields, almost real closed fields.}

\begin{document}

\begin{abstract}
Cluckers and Lipshitz have shown that real closed fields equipped with real analytic structure are o-minimal. 
This generalizes the well-known subanalytic structure $\RR_{\an}$ on the real numbers.
We extend this line of research by investigating ordered fields with real analytic structure that are not necessarily real closed.
When considered in a language with a symbol for a convex valuation ring, these structures turn out to be tame as valued fields: we prove that they are $\omega$-h-minimal. 
Additionally, our approach gives a precise description of the induced structure on the residue field and the value group, and naturally leads to an Ax--Kochen--Ersov-theorem for fields with real analytic structure.
\end{abstract}

\maketitle

\section{Introduction}

Since the work of {\L}ojasiewic, Gabrielov, and Hironaka from the sixties, it has been known that subanalytic sets in $\RR^n$ exhibit tame behaviour similar to the semi-algebraic setting~\cite{Loj65,Gab68,Hir73,Hir73b}. In more modern terms, the structure $\RR$ expanded with function symbols for restricted analytic functions is o-minimal. The original works by {\L}ojasiewic, Gabrielov, and Hironaka are rather long and difficult, and work by Denef--van den Dries~\cite{DvdD} greatly simplifies the treatment of subanalytic sets through the use of \emph{Weierstrass division}. Since then, this Weierstrass division has been the central tool for understanding and developing analytic structures on fields, see e.g.\ ~\cite{vdDAx, DMM,DenLip,CLRr,CLR,CLips,CubidesHaskell,Lip93, bhardwaj-vdd}. A comprehensive treatment of fields with analytic structure is the foundational work by Cluckers--Lipshitz~\cite{CLip}, which includes both the case of real closed fields and henselian valued fields with analytic structure.

In the current article we continue this line of research, by focussing on fields equipped with real analytic structure, as defined in~\cite{CLip}. Up to now, these structures have only been considered when the field is real closed, leading to o-minimal structures. However, the notion of a field with real analytic structure makes sense even without this assumption, and the aim of this article is to develop the theory of arbitrary fields equipped with real analytic structure. As it turns out, any such field comes with a natural henselian valuation for which the resulting residue field is $\RR$. In particular, such fields are \emph{almost real closed}. 

The prototypical example to keep in mind is the field $K = \RR((t))$ with its usual valuation. Then any analytic function on $[-1,1]_{\RR}^n$ defined by some convergent power series $f(x_1, \ldots, x_n)$ still makes sense when evaluating at elements of $[-1,1]_K^n$. This equips $K$ with \emph{real analytic structure}. Note that this structure contains more than just the $t$-adically convergent power series. In full generality, we study fields equipped with $\cB$-analytic structure, where $\cB$ is a real Weierstrass system as defined by Cluckers--Lipshitz~\cite{CLip}. In fact, for our purpose we need to amend the definition given there, as will be explained in detail in Section~\ref{sec:weierstrass.systems}. We note that all natural examples of real Weierstrass systems are still included in our updated definition.

Our first main result is that almost real closed fields equipped with $\cB$-analytic structure exhibit tame behaviour. Of course, in this generality one cannot hope for the theory to be o-minimal or even weakly o-minimal, simply because the field is not necessarily real closed. Instead we show that such fields are $\omega$-h-minimal, a tameness notion for valued fields developed by Cluckers--Halupczok--Rideau~\cite{CHR}.

As a consequence of $\omega$-h-minimality, one automatically obtains that definable functions satisfy the Jacobian property, cell decomposition results, and dimension theory. Furthermore, one also obtains bounds on rational points on transcendental definable curves in a suitable sense~\cite{CNSV}, see Example~\ref{ex:counting.dim}. This counting of rational points was in fact the original motivation for this project.

Denote by $\cL_{\val, \cB}$ the language of valued fields $\{0,1,+,\cdot, \cO\}$ expanded with function symbols for all elements of $\cB$. Then we prove the following.

\begin{theorem}\label{thm:real.h.minimal}
Let $\cB$ be a real Weierstrass system which is strong and rich, and let $K$ be an almost real closed field with $\cB$-analytic structure. Let $\cO_K$ be a convex valuation ring of $K$ and consider $K$ as an $\cL_{\val, \cB}$-structure. Then $\Th_{\cL_{\val, \cB}}(K)$ is $\omega$-h-minimal.
\end{theorem}

The terminology of rich Weierstrass systems will be explained in Section~\ref{sec:weierstrass.systems}, but let us point out that all examples of real Weierstrass systems from~\cite{CLip} are rich. 

Our proof strategy consists of two main steps, and is similar in spirit to that of $\omega$-h-minimality for valued fields with separated analytic structure~\cite[Thm. 6.2.1]{CHR}. The first step is a precise analysis of 1-terms, where there is a subtle, yet significant difference compared to~\cite{CHR}. Recall that a main input for \cite[Thm. 6.2.1]{CHR} is a good understanding of 1-terms with parameters in arbitrary models. This follows from an analysis of 1-terms without parameters and a procedure to enlarge the Weierstrass system with constants from the valued field. The latter can be achieved by finding a larger Weierstrass system in which both the valued field and the original Weierstrass system embed, in a compatible way, as in \cite{CLip}. In our setting, the existence of such a larger system is much less clear, unless our field is a Hahn series equipped with the natural analytic structure (Example \ref{ex:structure-natural}).

The Embedding Theorem~\ref{thm:embedding} is our solution to this problem. Essentially, it allows us to reduce to the nice situation mentioned above, and allows one to develop extension of parameters for real Weierstrass systems. We then build on results from \cite{CLRr} by Cluckers, Lipshitz and Robinson for 1-terms without parameters in real closed fields. Since we are working more generally with fields which are not real closed, this requires some work and care. Since we keep precise track of parameters in each step, we are able to first investigate terms over the real closure and then descend this analysis to the original field. This analysis of terms takes up all of Section 3.

The second step of our proof consists of a relative quantifier elimination statement down to the leading term structure. Here, we follow classical methods by Denef--van den Dries~\cite{DvdD} to reduce to the algebraic case, where quantifier elimination relative to the leading term sorts is well-studied (see e.g.\ ~\cite{Basarab,Kuh94,Flen}). This is carried out in Section~\ref{sec:QE}. 
%

The notion of $\omega$-h-minimality is really a tameness notion relative to the leading term structure $\RV$. As such, the above theorem does not readily give any information about the induced structure on the residue field and the value group. Therefore, our second main result gives a precise description of the induced structure on the residue field and the value group when $K$ is equipped with its natural valuation. We state the following as an exemplar result, but see Theorem~\ref{thm:induced.structure} for more general Weierstrass systems, as defined in Definition~\ref{def:real.weierstrass.system}. We work in a certain three-sorted language $\cL_{\ac, \cB}$ which includes an angular component map $\ac$, see Section~\ref{sec:QE.RF.VG} for a precise description. Denote by $\RR_\an$ the structure $\RR$ in the ring language expanded by symbols for restricted analytic functions.

\begin{theorem}[{{Theorem~\ref{thm:induced.structure}}}]\label{thm:induced.structure.intro}
Let $\cB = \cA((\Omega))$ be the full Weierstrass system over some ordered abelian group $\Omega$, and let $K$ be an almost real closed field with $\cB$-analytic structure. We consider $K$ as a valued field with its natural valuation. Then
\begin{enumerate}
	\item $\Th_{\cL_{\ac, \cB}}(K)$ is $\omega$-h-minimal,
	\item the definable subsets in the residue field are those definable in $\RR_{\an}$,
	\item the definable subsets in the value group are those definable in the language of ordered abelian groups,
	\item the residue field and value group are stably embedded and orthogonal.
\end{enumerate}
\end{theorem}

For more general Weierstrass systems $\cB$, one will have a different structure on the residue field in which there may be less analytic functions. In short, this proof follows from a relative quantifier elimination result in the language $\cL_{\ac, \cB}$. 

From this relative quantifier elimination, we also obtain an Ax--Kochen--Ersov type result. If $K$ and $K'$ are almost real closed fields considered with the natural valuation, then the classical Ax--Kochen--Ersov theorem tells us that $K$ and $K'$ are elementarily equivalent as valued fields if and only if their value groups are elementarily equivalent. Indeed, the residue fields are both equal to $\RR$. The next result can be seen as an extension of this, where we equip $K$ and $K'$ with $\cB$-analytic structure. Denote by $\cL_{\oag}$ the language of ordered abelian groups $\{0,+,<\}$.

\begin{theorem}[{{Theorem~\ref{thm:AKE}}}]\label{thm:AKE.intro}
Let $\cB$ be a real Weierstrass system which is strong and rich, and let $K$ and $K'$ be almost real closed fields with $\cB$-analytic structure. We consider $K$ and $K'$ as valued fields with their natural valuation, and denote their value groups by $G$ and $G'$.
Then 
\[
K\equiv_{\cL_{\ac, \cB}} K' \text{ if and only if } G\equiv_{\cL_{\oag}} G'.
\]
\end{theorem}

\subsection{Acknowledgements} The authors thank Raf Cluckers, Pierre Touchard, and Neer Bhardwaj for interesting discussions related to this paper. We thank Neer Bhardwaj for sharing a preprint around analytic Ax--Kochen--Ersov theory. 

The author F.V.\ is supported by F.W.O.\ Flanders (Belgium) with grant number 11F1921N. The author K.H.N\ is supported by F.W.O.\ Flanders (Belgium) with grant number 1270923N.

\section{Preliminaries}

\subsection{Weierstrass systems}\label{sec:weierstrass.systems}

We recall and adapt the notion of real Weierstrass systems, and fields equipped with real analytic structure. Let $n\geq 0, \alpha\in \RR_{>0}$ and let $A_{n, \alpha}$ be the ring of real power series in $\RR[[\xi_1, \ldots, \xi_n]]$ with radius of convergence strictly larger than $\alpha$. Recall that an element $f\in A_{n, \alpha}$ is \emph{regular in $\xi_1$ of degree $s$ at $0$} if we can write $f(\xi_1, 0, \ldots, 0, 0) = a\xi_1^s + \text{higher degree terms}$, where $a\neq 0$. Let $\Omega$ be some non-zero ordered abelian group. We define $\RR(\Omega)$ to be the field of rational functions in $t^\omega, \omega\in \Omega$, we define
\[
\RR((\Omega)) = \left\{ \sum_{i\in I} a_it^i \mid a_i\in \RR, I\subset \Omega \text{ well-ordered}\right\},
\]
and
\[
A_{n,\alpha}((\Omega)) = \left\{ \sum_{i\in I} f_it^i \mid f_i\in A_{n, \alpha}, I\subset \Omega \text{ well-ordered}\right\}.
\]
Note that $A_{0, \alpha}((\Omega)) = \RR((\Omega)) \subset A_{n, \alpha}((\Omega))$. We use the notation $\norm{\cdot}$ for the Gauss norm on $\RR((\Omega))$, where
\[
\bignorm{\sum_{i\in I}a_it^i} = t^{i_0}
\]
where $i_0 = \min \{i\mid i\in I, a_i\neq 0\}$. This norm extends naturally to $A_{n, \alpha}((\Omega))$ and will also be denoted by $\norm{\cdot}$. If $\norm{\sum_{g\in I} f_gt^g} = t^{g_0}$ then we call $f_{g_0}$ the \emph{top slice} of $f$. If $\norm{f} = 1$, then we call $f$ \emph{regular in $\xi_1$ of degree $s$ at $0$} if $f_0$ is regular in $\xi_1$ of degree $s$ at $0$.

We recall the notion of a real Weierstrass system from~\cite[Def.\,3.1.1]{CLip}. Our definition is in fact slightly different from~\cite{CLip}, and the differences will be explained in more detail below.

\begin{defn}\label{def:real.weierstrass.system}
Let $\cB = \{B_{n,\alpha} \mid n\in \NN, \alpha \in \RR_{>0}\}$ be a family of $\RR$-algebras with
\[
\RR[\xi_1, \ldots, \xi_n]\subset B_{n, \alpha} \subset A_{n, \alpha}((\Omega)).
\]
The family $\cB$ is a \emph{real Weierstrass system (over $\Omega$)} if the following hold:
\begin{enumerate}
\item \label{it:ax_pt1} \begin{enumerate}
	\item \label{it:ax_radii} If $m\leq m'$ and $\alpha'\leq \alpha$ then $B_{m, \alpha}\subset B_{m', \alpha'}, B_{0, \alpha} = B_{0, \alpha'} =: B_0$.
	\item \label{it:ax_coeff} If $f\in B_{m+n, \alpha}$ and $f = \sum_{\mu} f_\mu(\xi_1, \ldots, \xi_m)\eta^\mu$, where $\eta = (\xi_{m+1}, \ldots, \xi_{m+n})$, then the $f_\mu$ are in $B_{m, \alpha}$.
	\item \label{it:ax_rescaling} If $f\in B_{m, \alpha}, a\in (-\alpha, \alpha)^n\cap \RR^n$, and $r\in \RR_{>0}$, then $f(r\cdot \xi+a)$ is an element of $B_{m, \delta}$ with $\delta = \min\left\{ \frac{\alpha-a}{r}, \frac{\alpha+a}{r}\right\}$. \label{it:rescale}
	\item \label{it:ax_norm} If $f\in B_{n, \alpha}$ then there is some $a\in B_0$ such that $af\in B_{n, \alpha}$ satisfies $\norm{af} = 1$.
	\item \label{it:ax_permutation} If $f\in B_{n, \alpha}$ and $\sigma: \{1, \ldots, n\}\to \{1, \ldots, n\}$ is a permutation then $f(\xi_{\sigma(1)}, \ldots, \xi_{\sigma(n)})$ is in $B_{n, \alpha}$.
	\end{enumerate}
\item \emph{Weierstrass division:} \label{it:ax_Wdiv} If $f\in B_{n, \alpha}$ with $\norm{f} = 1$ is regular in $\xi_1$ of degree $s$ at $0$, then there is a $\delta\in \RR_{>0}$ such that for every $m\geq n$ and every $g\in B_{m, \alpha}$ there exist unique $Q\in B_{m, \delta}$ and $R_0, \ldots, R_{s-1}\in B_{m-1, \delta}$ with $\norm{Q}, \norm{R_i}\leq \norm{g}$ such that
\[
g = Qf + R_0(\xi') + R_1(\xi')\xi_1 + \ldots + R_{s-1}(\xi')\xi_1^{s-1},
\]
where $\xi' = (\xi_2, \ldots, \xi_m)$, and where $f$ is considered in $B_{m, \alpha}$ via the inclusion $B_{n, \alpha}\subset B_{m, \alpha}$.
\end{enumerate}

If additionally $\cB$ satisfies the following condition, then we say that $\cB$ is a \emph{strong Weierstrass system}
\begin{enumerate}
\item[(3)] \label{it:ax_strong} If $f(\xi, \eta_1, \eta_2)\in B_{n+2, \alpha}$ then there are $f_1(\xi, \eta_1, \eta_3), f_2(\xi, \eta_2, \eta_3)$ and $Q(\xi, \eta_1, \eta_2, \eta_3)\in B_{n+3, \alpha}$ such that 
\[
f(\xi, \eta_1, \eta_2) = f_1(\xi, \eta_1, \eta_3) + \eta_2 f_2(\xi, \eta_2, \eta_3) + Q\cdot (\eta_1 \eta_2 - \eta_3).
\]
\end{enumerate}

If additionally $\cB$ satisfies the following condtion, then we say that $\cB$ is a \emph{rich Weierstrass system}
\begin{enumerate}
	\item[(4)] \label{it:ax_rich} $\RR(\Omega)$ is a subring of $B_0$.
\end{enumerate}
Denote by $B_{n,\alpha}^{\circ}$ the ring of $f \in B_{n,\alpha}$ such that $\norm{f}  \leq 1$ and by $B_{n,\alpha}^{\circ \circ} $ the ideal consisting of $f$ with $\norm{f} <1$. 
\end{defn}
\begin{example}\label{ex:full_Weierstrass}
Consider the system of algebras with $B_{n,\alpha} = A_{n,\alpha}((\Omega))$ for each $\alpha,n$.
This is a (rich and strong) real Weierstrass system~\cite[Rem.~3.13(iv)]{CLip}.
It is called the \emph{full Weierstrass system over $\Omega$} and denoted by $\cA((\Omega))$. 
\end{example}

Our definition differs in two ways from~\cite{CLip}, so let us indicate the changes. 
First, axiom 1(e) is new and simply tells us that we can permute the variables of the elements of a Weierstrass system. 
This seems to be necessary to ensure that composition is well-defined and is also present in the analogous definition of separeted Weierstrass system \cite[Def.~A.1.1]{CLips}. Second, in the Weierstrass division axiom the value $\delta\in \RR_{>0}$ does not depend on $m$, the number of variables of $g$. This will be important for us when expanding Weierstrass systems in Section~\ref{sec:extending.weierstrass}. The richness condition is new, and will be needed for us when embedding a field with real analytic structure into a Hahn series field. 

The notion of a strong Weierstrass system is also present in~\cite[Def.\,3.1.1]{CLip} and will be needed in Section~\ref{sec:term.real.closed} when analysing terms on the real closure. It roughly allows one to write $f(x,1/x) = g(x) + (1/x) h(1/x)$.

It is important to note that all natural examples of real Weierstrass systems from~\cite{CLip} satisfy the axioms above, in particular they are also real Weierestrass systems according to our definition. Indeed, in the Weierstrass division axiom the value radius $\delta$ depends only on $f$, and not on $g$, as follows from~\cite[Thm.\, ~II.1, p.\ ~80]{gunning_rossi}. Additionally, all of these examples satisfy the richness condition (\ref{it:ax_rich}). Since our notion of a real Weierstrass system is more restrictive, all results of~\cite{CLip} about them still hold in our context.

We also need the notion of a field with real analytic structure. 

\begin{defn}\label{def:real.analytic.structure}
Let $\cB$ be a real Weierstrass system and let $K$ be an ordered field containing $B_0$ as an ordered subfield. For each $n\in \NN, \alpha\in \RR_{>1}$ let $\sigma_{n, \alpha}$ be an $\RR$-algebra morphism from $B_{n, \alpha}$ to the ring of $K$-valued functions on $[-1,1]_K^n$, compatible with the inclusions $B_{n, \alpha}\subset B_{n, \beta}$ when $\beta<\alpha$ and respecting condition (1)(c) from Definition~\ref{def:real.weierstrass.system}. Let $\sigma_n$ be the induced map on $\cup_{\alpha>1} B_{n, \alpha}$. Assume that the maps $\sigma_n$ satisfy
\begin{enumerate}
\item $\sigma_0$ is the inclusion $B_0\to K$, 
\item $\sigma_m(\xi)$ is the $i$-th coordinate function on $[-1,1]_K^m$,
\item $\sigma_{m+1}$ extends $\sigma_m$, where we identify functions on $[-1, 1]_K^m$ with functions on $[-1,1]_K^{m+1}$ that do not depend on the last variable,
\item $\sigma_m$ respects permutation of coordinates.
\end{enumerate}
Then we say that $K$ \emph{has real analytic $\cB$-structure (via $\sigma = \{\sigma_{n, \alpha}\}$)}.
\end{defn}
\begin{example}[{\cite[Section 3.3]{CLip}}] \label{ex:structure-natural}
\begin{enumerate}
	\item For $\Omega = \{0\}$ and $B_{n,\alpha} =  A_{n,\alpha}((0)) = A_{n,\alpha}$, we have a natural $\cB$-analytic structure on $\RR$, which is just the subanalytic structure $\RR_{\an}$.
	\item If $\Omega \leq G$, then for each Weierstrass system $\cB$ over $\Omega$, $\RR((G))$ carries a canonical $\cB$-analytic structure, as considered in \cite{DvdD}.
	Evaluation of some $f(x) \in B_{n,\alpha}$ at $a \in \RR((G))^n$ can be understood either as substituting $x$ by $a$ in the power series $f(x) \in \RR((G))[[x]]$, or equivalently via iterated Weierstrass division by the $(x_i - a_i)$ in the full Weierstrass system $\cA((G))$ (which naturally contains $\cB$). We refer to this as the \emph{natural} $\cB$-analytic structure on $\RR((G))$.
	\item For $\Omega = \QQ$ and
	\[ B_{n,\alpha} = \{ \sum_{i} t^{\gamma_i} f_{\gamma_i} \mid  f_{\gamma_i} \in A_{n,\alpha}  ,\QQ \ni \gamma_i \to \infty \}, \]
	the completion of the field of real Puiseux series has a natural $\cB$-analytic structure~\cite{CLRr}.
\end{enumerate}
\end{example}

We collect some general facts about Weierstrass systems and ordered fields with $\cB$-analytic structure.
For the remainder of this section, let $K$ be an ordered field with real analytic structure.
Since $B_0$ contains $\RR$ and is contained in $K$ via the analytic $\cB$-structure, the field $\RR$ naturally sits inside $K$.

\begin{lem} \label{le:unit}
If $U(x) \in B_{m,\alpha}$ such that $\norm{U} = 1$ and $U_0(0) \neq 0$, then there exists some $\delta > 0$ such that $U(x) \in B_{m,\delta}^{\times}$ 
\end{lem}
\begin{proof}
This is part of \cite[Rem.~3.1.4(ii)]{CLip} and follows by Weierstrass dividing $1$ by $U(x)$, noting that the latter is regular of degree zero (in any variable).
\end{proof}
The following lemma is a real version of the implicit function theorem from \cite[Rem.~1.3(4)]{DenLip} and is proved in a similar manner.
\begin{lem} \label{lem:implicit}
Let $f(x,y) \in B_{m+1,\alpha}$ for certain $\alpha,m$ such that $\norm{f} = 1$ and $f(x,y)$ is regular of degree 1 in $y$.
Then there exist some $\delta \in \RR_{>0}$ and $r(x) \in B_{m,\delta}$ such that the composition $f(x,r(x))$ is well-defined in $B_{m,\delta}$ and $f(x,r(x)) = 0$.
\end{lem}
\begin{proof}
We follow the proof of \cite[Rem.~1.3(4)]{DenLip}.
By Weierstrass division, we may write
\begin{equation} \label{eq:implicit}
	y = f(x,y) U(x,y) + r(x)
\end{equation}  
inside some $B_{m+1,\delta}$.
Comparing the coefficients of $y$ in the above equality of power series, we find that $U_0(0,0) \neq 0$, in particular $U(x,y)$ is a unit by Lemma \ref{le:unit}, up to shrinking $\delta$.
Then, by comparing constant terms on both sides, it follows that $r_0(0) = 0$.
This ensures that, up to further shrinking $\delta$, that $f(x,r(x))$ and $U(x,r(x))$ are well-defined elements of $B_{m,\delta}$.
Since the latter is a unit, it follows from (\ref{eq:implicit}) that $f(x,r(x)) = 0$.
\end{proof}
The proof idea of the lemma below is classical, see e.g.\ ~\cite[Lem.~2.4]{DMM}.
In this more abstract context, some additional input from Lemma~\ref{lem:implicit} is required, as well as some care about radii of convergence.
\begin{lem} \label{le:order-preserving}
Let $f(x) \in B_{m,\alpha}$ with $\alpha > 1$ and such that $\norm{f} = 1$ and $f_0(0) > 0$.
There exists some $\delta \in \RR_{>0}$ such that if $K$ is an ordered field with $\cB$-analytic structure then $\sigma(f)(a) > 0$ for all $a$ in $[-\delta,\delta]_K$.
\end{lem}
\begin{proof}
Take $\varepsilon \in \RR_{>0}$ and consider the polynomial $h_{\varepsilon}(x,y)  = x + \varepsilon^2 - (y + \varepsilon)^2$.
From Lemma \ref{lem:implicit}, we deduce the existence of some $r_{\varepsilon}(x) \in B_{1,\varepsilon}$ such that $(r_{\varepsilon}(x) + \varepsilon)^2 = x + \varepsilon^2$, as an element of some $B_{1,\delta}$.
Now fix a $\varepsilon \in \RR_{>0}$ such that $f_0(0) > \varepsilon^2$ and let $q(x) = r_{\varepsilon}(x) + \varepsilon $. 
For $g(x) = f(x) - f(0)$, the composition $q(g(x))$ is a well-defined element of $B_{1,\delta}$, up to shrinking $\delta$.
Now $q(g(x/(2 \delta))) \in B_{1,2}$ and by construction, it holds for all $a \in [-1,1]_K$ that
\[ (\sigma(q \circ g \circ (x/(2 \delta)) (a))^2 = \sigma(f) (a/(2\delta)) - \sigma(f(0))  +  \varepsilon^2  .\]
Since squares in $K$ are non-negative, it follows that $\sigma(f)(a/(2\delta)) \geq \sigma(f(0)) - \varepsilon^2$.
As $\sigma \colon B_0 \to K$ preserves the order and $f(0) > \varepsilon^2$ by construction, it follows that $\sigma(f)$ takes only strictly positive values on $[-\delta/2,\delta/2]$.
\end{proof}
The above lemma implies that the interpretation map $\sigma$ is continuous with respect to the Gauss norm on $\cB$ and the supremum norm on $K$.
\begin{lem} \label{le:automatic_continuity}
Let $\alpha \in \RR_{>1}$.
For any $f(x) \in B_{m,\alpha}$ with $\norm{f} \leq 1$ it holds that $\norm{(\sigma f)(a) } \leq 1$, for all $a \in [-1,1]_K$.
\end{lem}
\begin{proof}
Since $\sigma \colon B_0 \to K$ is order-preserving, (\ref{it:ax_norm}) allows us to reduce to the case that $\norm{f} = 1$.
Using axiom (\ref{it:rescale}) and compactness of $[-1,1]_{\RR}$, it suffices to prove this lemma for $a \in [-\delta, \delta]_K$ for some $\delta\leq \alpha$ with $\delta \in \RR_{>0}$.
Now, there exist some $b,c\in \RR$ such that $b < f_0(0) < c$.
By Lemma \ref{le:order-preserving}, applied to $f(x)-b$ and $f(x) -c$, it follows that also $b < \sigma(f)(a) < c$ on some interval $a\in[-\delta, \delta]_K$, and thus $\norm{(\sigma f)(a)} = 1$. 
\end{proof}

We consider $K$ as a valued field whose valuation ring $K^\circ$ is the convex closure of $\RR$. 
Denote the maximal ideal of this valuation by $\Kinf$, we call elements of $\Kinf$ \emph{infinitesimal}. The corresponding valuation $\norm{\cdot } \colon K^{\times} \mapsto G$ will be called the \emph{natural valuation}, and will be written multiplicatively. Note that the residue field of $K$ is isomorphic to $\RR$, and that the reduction map $\Kfin\to \Kfin/\Kinf$ restricts to an isomorphism on $\RR\subset K$. For an element $x\in K^\circ$, denote by $x^\circ$ the element of $\RR\subset K$ which is closest to $x$, i.e.\ for which $x-x^\circ\in \Kinf$. The following lemma tells us that this valuation is automatically henselian. This is again a variation on a classical argument, see e.g. \cite[Lem.~2.5]{DMM}, in our more abstract context.

\begin{lem}\label{lem:cB.implies.henselian}
Any ordered field with $\cB$-analytic structure is henselian with respect to the natural valuation.
\end{lem}
\begin{proof}
Let $K$ be an ordered field with $\cB$-analytic structure. It suffices to show that any polynomial $p(y) \in K[y]$  of the form
\[ p(y) = 1 + y + a_2 y^2 + \dots + a_d y^d, \]
with all $a_i \in K^{\circ \circ}$ has a root in $K^{\circ \circ}$ (see e.g. \cite[Lem.~2.2]{vdD14}).
Replace each $a_i$ by a variable $x_i$, to obtain $q(x,y) \in \ZZ[x,y]$ such that $q(a,y) = p(y)$, where $a = (a_2,\dots,a_d)$.
Now consider the polynomial $\tilde{q}(x,y) = q(x-1,y)$ and remark that it satisfies the hypotheses of Lemma \ref{lem:implicit}.
Let $r(x)$ be the implicit function thus defined.
Then $(\sigma r)(a)$ is a zero of $p(y-1)$, and moreover $\norm{(\sigma r)(a)} \leq 1 $ by Lemma \ref{le:automatic_continuity}, using that $\norm{r}\leq 1$.
\end{proof}

We will need to extend analytic structures to algebraic field extensions, for which we use the following.

\begin{lem} \label{lem:alg-extension}
Let $L$ be an ordered algebraic field extension of $K$, then there exists a unique $\cB$-analytic structure on $L$ extending the one on $K$.
\end{lem}
\begin{proof}
Since the natural valuation on $K$ is henselian, it extends uniquely to the algebraic closure $K^\alg$ of $K$. In particular, this extension coincides on $L$ with the natural valuation on $L$.

This follows by Weierstrass division, as in e.g. \cite[Thm.~2.18]{CLR}.
We sketch the general idea.
Take an element $f\in B_{m, \alpha}$ and $a\in L^m$, we wish to evaluate $f(a)$. 
By axiom (\ref{it:ax_rescaling}) and compactness of $[-1,1]_{\RR}$, it suffices to consider the case where $a \in (L^{\circ \circ})^m$. 
Let $p_1(x,y), \dots, p_m(x_m,y) \in \ZZ[x,y]$ with $x = (x_1,\dots,x_m)$ and $y = (y_1,\dots,y_\ell)$ (for some $\ell \in \NN$), and $b \in K^\ell$ be such that $p_i(x_i,b)$ is the minimal polynomial of $a_i$.
By henselianity, all conjugates of each $a_i$ are infinitesimal in $K^\alg$, and hence all $b_j$ belong to $K^{\circ \circ}$ for $j=1, \ldots, \ell$.
By iterated Weierstrass division of $f(x)$ by $p_i(x_i,y)$ we obtain
\[
f(x) = p_1(x_1,y)  g_1(x,y) + \dots + p_m(x_m,y)  g_m(y)  + r(x,y),
\]
inside $B_{m+n, \delta}$, where $\delta \leq \alpha$, and $r(x,y)$ is a polynomial in each $x_i$. 
Hence if the analytic structure extends to $L$ then the value of $(\sigma f)(a)$ must be equal to $(\sigma r)(a,b)$, which is polynomial in the $a_i$ and thus uniquely determined by the $\cB$-analytic structure on $K$.
Conversely, expanding $\sigma$ in this way equips $L$ with $\cB$-analytic structure.
\end{proof}

\subsection{Almost real closed fields}

We recall the notion of almost real closed fields, and refer to~\cite{delon-farre, krapp} for more information.

\begin{defn}
An \emph{almost real closed field $K$} is an ordered field which admits a henselian valuation for which the residue field is real closed.
\end{defn}

\begin{remark}
Some authors define an almost real closed field as a field which admits a henselian valuation with real closed residue field. Any such field admits an order compatible with the residue field, by the Baer--Krull theorem. For us, the order is part of the data of an almost real closed field, but the henselian valuation is not.
\end{remark}

Almost real closed fields always come equipped with the natural valuation, for which the valuation ring is the convex closure of $\QQ$. We will more generally consider convex valuations.

\begin{defn}
If $K$ is an almost real closed field, then a \emph{convex valuation} on $K$ is a valuation whose valuation ring is a convex set with respect to the ordering on $K$.
\end{defn}

One can show that on an almost real closed field, a valuation is convex if and only if it is henselian, see~\cite[Lem.\,2.1]{KW76} and~\cite[Prop.\,2.9]{delon-farre}. In particular, an ordered field is almost real closed if and only if the natural valuation is henselian and has real closed residue field.

Fields with real analytic structure are automatically almost real closed, as the following lemma shows.

\begin{lem}\label{lem:cB.implies.almost.real.closed}
Any ordered field with $\cB$-analytic structure is almost real closed.
\end{lem}

\begin{proof}
The natural valuation has residue field $\RR$ and is henselian by Lemma~\ref{lem:cB.implies.henselian}.
\end{proof}

\subsection{Hensel minimality}

In this section we recall the definition of Hensel minimality in the equicharacteristic zero setting~\cite{CHR}. This section is unrelated to the previous ones. We first recall the set-up, and some notation for valued fields. 

Let $K$ be a valued field with valuation ring $\cO_K$ and value group $\Gamma_K$. We denote the valuation additively by $v_K: K\to \Gamma_K\cup \{\infty\}$. The maximal ideal of $K$ is denoted $\cM_K$ and the residue field by $k = \cO_K/\cM_K$. For $a\in K, \lambda\in \Gamma_K$ let $B_{> \lambda}(a)$ and $B_{\geq \lambda}(a)$ be the open resp.\ closed ball of radius $\lambda$ with centre $a$. For $\lambda\in \Gamma_{K,\geq 0}$, we define 
\[
\RV_\lambda = \frac{K^\times}{1+B_{>\lambda}(0)} \cup \{0\},
\]
and the natural map $K\to \RV_{\lambda}$ by $\rv_{\lambda}$. If there are multiple fields around, we will denote this by $\RV_{K, \lambda}$. We let $\RV_\lambda^\times = \RV_\lambda\setminus \{0\}$, and denote $\RV_0$ and $\rv_0$ by $\RV$ and $\rv$. The structure $\RV$ combines information from the residue field and the value group, via the exact sequence
\[
1\to k^\times \to \RV^\times \to \Gamma_K\to 0.
\]

Let $c$ be an element of $K$ and $\lambda\in \Gamma_{K,\geq 0}$, the \emph{balls $\lambda$-next to $c$} are precisely the fibres of the map $x\mapsto \rv_\lambda(x-c)$. If $C\subset K$ is a finite set, then a \emph{ball $\lambda$-next to $C$} is a (non-empty) intersection of balls $B_c$ for $c\in C$, where $B_c$ is $\lambda$-next to $c$. Say that a finite set $C$ \emph{$\lambda$-prepares} a set $X\subset K$ if each ball $\lambda$-next to $C$ is either contained in $X$, or disjoint from $X$.

We can now define what h-minimality is, following~\cite{CHR}.

\begin{defn}[{{\cite[Def.\,1.2.3]{CHR}}}]
Let $\cL$ be a language expanding the language of valued fields $\cL_{\val} = \{0,1,+,\cdot,\cO\}$, and let $T$ be a complete theory of $\cL$-structures which are equicharacteristic zero henselian valued fields. Let $\ell$ be a positive integer (including $0$) or $\omega$. Then we say that $\cT$ is \emph{$\ell$-h-minimal} if the following holds for every model $K$ of $T$:
\begin{enumerate}
\item[] For every $\lambda\in \Gamma_{K,\geq 0}$, every $A\subset K$, and every $A'\subset \RV_{K, \lambda}$ with $|A'|\leq \ell$, if $X\subset K$ is $\cL(A\cup \RV_K\cup A')$-definable, then there exists a finite $\cL(A)$-definable set $\lambda$-preparing $X$.
\end{enumerate} 
\end{defn}

The notion of h-minimality implies several tameness properties about definable objects, and is preserved under various natural operations, see e.g.\ ~\cite{CHR, CHRV, Verm:h-min}. We will discuss this in more detail in Section~\ref{sec:h.min.proof}, where we will apply it to almost real closed fields with analytic structure.

\subsection{Notation}

For convenience of the reader, we list here some recurring notation in the following sections.

\begin{longtable}{p{2.5cm}p{10cm}}
$K$ & an almost real closed field \\
$K^\rcl$ & the real closure of $K$ \\
$|\cdot |: K\to K_{>0}$ & the absolute value on $K$ \\
$\norm{\cdot}: K\to G$ & the natural valuation on $K$, with multiplicative notation \\
$K^\circ, K^{\circ \circ}$ & the valuation ring and maximal ideal for the natural  valuation on $K$ \\
$x^\circ, x\in K^\circ$ & the element $x^\circ$ in $\RR\subset K$ closest to $x$ \\
$v: K\to \Gamma$ & a convex valuation on $K$, with additive notation \\
$\cO_K, \cM_K$ & the valuation ring and maximal ideal of $v$\\

$\cB$ & a real Weierstrass system over $\Omega$, typically strong and rich \\
$\Omega$ & the value group of the Weierstrass system $\cB$ \\
$\norm{\cdot}$ & the Gauss norm on $\cB$ \\

$\cL_{\val}$ & the language of valued fields $\{0,1,+,\cdot, \cO\}$ \\
$\Lor$ & the language of ordered rings $\{0,1,+,\cdot, <\}$ \\
$\cL_{\cB}$ & the language $\Lor$, together with a symbol $(\cdot)^{-1}$ for field inversion, and function symbols for all elements of $\cB$ \\
$\cL_{\ovf}$ & the language of ordered valued fields $\{0,1,+,\cdot, \cO, <\}$ \\
$\cL_{\val, \cB}$ & the union of the languages $\cL_{\val}$ and $\cL_{\cB}$ \\
$\cL_{\RV}$ & two-sorted language with $\cL_{\val}$ on the valued field sort and $\{0,1,\cdot, \oplus\}$ on the $\RV$-sort \\
$\cL_{\RV_\lambda, \cB}$ & two-sorted language with $\cL_{\val, \cB}$ on the valued field sort and the full induced structure on $\RV_\lambda$ \\
$\cL_{\ac, \cB}$ & three-sorted language with symbols for $\cB$ and an angular component map $\ac$

\end{longtable}

\section{Analysing terms}\label{sec:term.analaysis}

Let $\Omega$ be an ordered abelian group, and let $K$ be an ordered field equipped with real analytic $\cB$-structure, for some real Weierstrass system $\cB$ over $\Omega$. Crucial for us is a precise control of the $1$-terms \emph{with} parameters from $K$. Our approach to analyse these $1$-terms is as follows. First, to allow for parameters, we need to extend our Weierstrass system to a larger one where we include the parameters. Secondly, we prepare the desired term on the real closure of $K$, following the methods of~\cite{CLRr}. Finally, we pull back this preparation to the ground field $K$. These steps are achieved in the following three subsections, with Proposition~\ref{prop:term_analysis} being the main tool from this section. 

Let $\cL_\cB$ be the language $\{0,1,+,\cdot, (\cdot)^{-1}, <\}$ together with function symbols for all elements of $\cB$. Then the field $K$ is naturally an $\cL_\cB$-structure. Even though we will consider $K$ as a valued field for the natural valuation $\norm{\cdot}: K^\times\to G$, note that $\cL_\cB$ does not have a symbol for the valuation ring.

\subsection{An embedding theorem}\label{sec:embedding.theorem}
A key step in our approach is to first embed any field $K$ with $\cB$-structure into a Hahn series field $\RR((G))$ with the \emph{canonical} $\cB$-structure from Example \ref{ex:structure-natural}. This allows us to easily add parameters from $K$ to the elements from $\cB$. This adding of parameters is explained in more detail in Section~\ref{sec:extending.weierstrass}.

The following lemma is probably well-known to experts, but as we could not find a suitably precise reference, we give some details. 

\begin{lem} \label{lem:extend_section}
Let $K$ be an almost real closed field, with natural valuation $\norm{\cdot} \colon K^{\times} \to G$. 
Then, for any subgroup $\Gamma \leq G$, any partial section $s_0 \colon \Gamma \to K_{>0}$ of the valuation extends to a section $s \colon G \to K_{>0}$.
\end{lem}
\begin{proof}
We first prove the following claim:
\begin{enumerate}
	\item[] Let $a\in K_{>0}$ and $n\in \NN$ such that $\norm{a}$ is an $n$-th power in $G$. Then $a$ has a unique $n$-th root in $K_{>0}$.
\end{enumerate}
For the proof of this claim, take $b\in K_{>0}$ such that $ \norm{b}^n = \norm{a}$ and put $u = a/b^n$. Then $u>0$ and $\norm{u} = 1$. Since the residue field is real closed, Hensel's lemma implies that $K$ contains a unique positive root $c$ of the polynomial $x^n - u$. But then $bc$ is the unique positive $n$-th root of $a$ in $K_{>0}$, proving the claim.

Now Zorn's lemma yields a maximal subgroup $H \leq G$ equipped with a partial section $s \colon H \to K_{>0}$ extending the one on $\Gamma$. Assume that $H \neq G$, we will show how to extend $s$ to a larger subgroup of $G$.

So take  $g\in G\setminus H$, we wish to extend $s$ to the group $H'$ generated by $H$ and $g$. By the claim, we may already assume that $G/H$ is torsion-free.
We simply take any $t_g\in K_{>0}$ for which $\norm{t_g} = g$. 
Since $G/H$ is torsion-free, it follows that every element of $H'$ can be written uniquely as $h\cdot g^n$ with $h\in H$ and $n\in \ZZ$. 
We extend $s$ to $H'$ via $s(h\cdot g^n) = s(h)t_g^n$, which is hence well-defined. 
It is straightforward to see that this is also a group morphism. We conclude that there exists a section $s: G\to K_{>0}$.
\end{proof}

For any language $\cL$ and $\cL$-structure on $K$, $a \in K$ and subset $A \subset K$ we denote by $\qftp_{\cL}(a/A)$ the set of all quantifier-free $\cL(A)$-formulas $\varphi(x)$ such that $\varphi(a)$ holds.
\begin{lem} \label{lem:qf-1-types}
Let $L$ be an $\cL_\cB$-substructure of $K$. Then, for each $a \in K$, $\qftp_{\cL_\cB}(a/L)$ is uniquely determined by $\qftp_{\Lor}(a/L)$.
\end{lem}
\begin{proof}
By Lemma \ref{lem:alg-extension} there is a unique extension of the $\cB$-analytic structure on $L$ and $K$ to their real closures $L^\rcl$ and $K^\rcl$.
Inside $K^{\rcl}$, the resulting $\qftp_{\cL_{\cB}}(a/L^{\rcl})$ clearly contains $\qftp_{\cL_{\cB}}(a/L)$.
Since the $\cL_{\cB}$-structure on $L^{\rcl}$ is o-minimal by \cite[Thm.~3.4.3]{CLip}, it follows that $\qftp_{\cL_{\cB}}(a/L^{\rcl})$ is determined by $\qftp_{\{<\}}(a/L^{\rcl})$.
By uniqueness of the real closure, the latter is completely determined by $\qftp_{\Lor}(a/L)$.
\end{proof}

The following embedding theorem can be seen as an extension of an embedding result by van den Dries--Macintyre--Marker~\cite[Cor.~3.6]{DMM}, and is the main result of this section.
Recall that if $\Omega$ is a a subgroup of an ordered abelian group $G$ and $\cB$ is a Weierstrass system over $\Omega$, then $\RR((G))$ has a natural $\cL_{\cB}$-structure, as described in Example \ref{ex:structure-natural}.
\begin{thm} \label{thm:embedding}
Let $\cB \subset \cA((\Omega))$ be a rich Weierstrass system and let $K$ be an ordered field with $\cB$-analytic structure, with value group $G$ for the natural valuation. Then there exists an inclusion $\Omega\hookrightarrow G$ of abelian groups, and an $\cL_{\cB}$-embedding $K\hookrightarrow \RR((G))$ where $\RR((G))$ is equipped with the natural $\cB$-analytic structure.
\end{thm}
\begin{proof}
The inclusions $\RR(\Omega) \subset B_0 \subset K$ determine an inclusion $\Omega \leq G$ and a partial section $s \colon \Omega \to B_0\subset K_{>0} \colon \omega \to t^{\omega}$ for the natural valuation on $K$. 
By Lemma \ref{lem:extend_section}, $s$ extends to a section $G \to K_{>0}$.
Additionally the inclusion $\Omega \leq G$ induces an $\cL_{\cB}$-embedding $\iota_0 \colon B_0 \hookrightarrow \RR((G))$.
Using Zorn's lemma, we construct a maximal $\cL_{\cB}$-substructure $L$ of $K$ together with an $\cL_{\cB}$-embedding $\iota \colon L \hookrightarrow \RR((G))$, extending the one on $B_0.$
We show that $L = K$.
Assume that $L \neq K$ and consider the following cases, where we write $\Gamma_L$ for the value group of $L$ with respect to the natural valuation:
\begin{enumerate}
	\item $\Gamma_L \neq G$ and $G/\Gamma_L$ has torsion, \label{it:cl-div}
	\item $\Gamma_L \neq G$ and $G/\Gamma_L$ is torsion-free, \label{it:cl-indep}
	\item $\Gamma_L = G$. 	\label{it:cl-final}
\end{enumerate}

We will consider $L$ as a subfield of $\RR((G))$ through $\iota$ (i.e.\ for $a \in L$, we will also write $a$ for its image $\iota(a) \in \RR((G))$).
We view $K$, $L$ and $\RR((G))$ as valued fields for their natural valuations $\norm{\cdot}$. 

(\ref{it:cl-div}) Take some $g \in G \setminus \Gamma_L$ and $h \in \Gamma_L$ such that $g^n = h$.
Then $s(g)$ is the unique positive $n$-th root of $s(h) \in L$.
Since Weierstrass systems extend uniquely to algebraic extensions by Lemma \ref{lem:alg-extension}, we can extend $\iota$ to $L(s(g))$ by mapping $s(g)$ to $t^g$, which is the unique positive $n$-th root of $t^h$.

(\ref{it:cl-indep}) Take again some element $g\in G\setminus \Gamma_L$ and write $s(g) = \lambda$.
Then, for all natural numbers $i \neq j$ and  $b,c \in L^{\times}$, we have $\norm{b \lambda^{i}} \neq \norm{c \lambda^{j}}$. 
In particular, if $p(x)  = \sum_{i =0}^n a_i x^i$, is a polynomial with coefficients in $L$, then there is a unique index $i_0$ such that $\norm{a_{i_0} \lambda^{i_0}}$ is maximal among  $\{\norm{a_i \lambda^i}\}_{i=0}^n$. 
It follows that $p(\lambda) > 0$ if and only if $a_{i_0} > 0$.
Since $i_0$ depends only on $g$ (and not on $\lambda$), it follows that $p(\lambda) > 0$ if and only if $p(t^g) > 0$.
Now Lemma \ref{lem:qf-1-types} allows us to extend $\iota$ on the substructure of $K$ generated by $L$ and $\lambda$, by mapping $\lambda$ to $t^g$, since $\qftp_{\Lor}(\lambda/L) = \qftp_{\Lor}(t^g/L)$.

(\ref{it:cl-final})
Since $L \neq K$, there exists some $\lambda \in \ifin{K} \setminus \ifin{L}$. 
Because $K$ is an immediate extension of $L$, we may find a sequence $(x_{\alpha})_{\alpha < \beta}$ in $L$ which pseudo-converges to $\lambda$ and which has no pseudo-limit in $L$~\cite[p.~104]{vdD14}. 
Now let $\mu \in \RR((G))$ be any a pseudo-limit of $(x_{\alpha})_{\alpha < \beta}$ in $\RR((G))$, which cannot be an element of $L$.
As in case (\ref{it:cl-indep}), it suffices to show that for any polynomial $p(x) \in L[x]$ we have $p(\lambda) > 0$ if and only if $p(\mu) > 0$.
Now notice that eventually $\norm{p(x_{\alpha}) - p(\lambda)} < \norm{ p(\lambda)}$, since the left hand side is eventually strictly decreasing and $\norm{p(x_{\alpha})}$ is eventually constant.
Indeed, $(x_{\alpha})_{\alpha < \beta}$ is of transcendental type, as $L$ is henselian of equicharacteristic zero, by Lemmas~\ref{lem:cB.implies.henselian} and~\cite[Corollary~4.22]{vdD14}. 
In particular, for all sufficiently large $\alpha$ we have $  \abs{p(x_{\alpha}) - p(\lambda)} <  \abs{ p(\lambda)} $. Hence the sign of $p(x_\alpha)$ is eventually constant, and equal to the sign of $p(\lambda)$. 
The same holds for $p(\mu)$, whence it has the same sign as $p(\lambda)$.
\end{proof}

\subsection{Extending real Weierstrass systems}\label{sec:extending.weierstrass}

Let $A$ be a subset of $K$. The goal of this section is to extend the Weierstrass system $\cB$ with elements from $A$, similar to \cite[Sec.~4.5]{CLip} for separated Weierstrass systems. The following proposition is the main goal.

\begin{prop}\label{prop:1.terms}
Let $\cB$ be a rich real Weierstrass system and let $K$ be an ordered field equipped with $\cB$-analytic structure. 
Given any $A$ subset of $K$, there exists a real Weierstrass system $\cB[A]$ and a $\cB[A]$-analytic structure on $K$ with the following property:
\begin{enumerate}
	\item[] For any $\cL_{\cB}(A)$ term $t_1(x)$ in any number of variables, there exists an $\cL_{\cB[A]}$-term $t_2(x)$ such that $K \models t_1(x) = t_2(x)$, and vice-versa.
\end{enumerate}
Moreover, if $\cB$ is a strong Weierstrass system then so is $\cB[A]$.
\end{prop}

The importance of the above proposition is that it allows one to understand terms with parameters from $A$ by understanding the Weierstrass system $\cB[A]$.
For the latter, we can use the tools developed in work by Cluckers--Lipshitz--Robinson~\cite{CLRr}.
First, we define how to extend parameters for substructures of $\RR((G))$. 
The general case can be reduced to this, via Theorem~\ref{thm:embedding}.
%
\begin{defn} \label{def:expanded-Weierstrass}
Let $\cB$ be a real Weierstrass system over $\Omega \leq G$.
Given an $\cL_\cB$-substructure $A$ of $\RR((G))$, we define a family of $\RR$-algebras $ \cB[A] = \{ B_{m,\alpha}[A]\}_{m,\alpha}$ where $m$ ranges over $\NN$ and $\alpha$ over $\RR_{>0}$: 
\[ B_{m,\alpha}[A] \coloneqq \{ f(x,\varepsilon) \mid f(x,y) \in B_{m+n,\alpha}, \varepsilon \in (A\cap \RR((G))^{\circ \circ})^n, n \in \NN   \} \subset \cA_{m,\alpha}((G)). \]
When $A$ is not a substructure, define $\cB[A]$ as $\cB[\langle A \rangle]$, where $\langle A \rangle$ is the $\cL_\cB$-substructure of $\RR((G))$ generated by $A$.
\end{defn}
\begin{lem} \label{lem:extended_Weierstrass}
The family of algebras $\cB[A] \subset \cA((G))$ as defined in Definition \ref{def:expanded-Weierstrass} is a real Weierstrass system.
Additionally, if $\cB$ is a strong real Weierstrass system, then so is $\cB[A]$.
\end{lem}

The following computation is the key technical tool for verifying Weierstrass division for $\cB[A]$.

\begin{lem} \label{lem:change_of_var}
Let $\cB$ be a real Weierstrass system over $G$. Suppose that $\varepsilon \in (\RR((G))^{\circ \circ})^n$ and $f(x,y) \in B_{m + n, \alpha}$ are such that  $\norm{f(x,\varepsilon)} = 1$. Then there exists some $\delta\leq \alpha$, $\ell \in \NN$, $g(x,z) \in B_{m + \ell, \delta}$ and some $h(y) \in (B_{n, \delta})^{\ell}$ such that $\norm{h(\varepsilon)} < 1$ and
\begin{equation} \label{eq:g-conditions}
	\left\{  
	\begin{array}{ll}
		\norm{g(x,z)} 			&= 1, \\
		g_0(x,0)				&= f_0(x,\varepsilon), \\
		g(x,h(\varepsilon))   	&= f(x,\varepsilon).
	\end{array} \right.
\end{equation} 
\end{lem}
\begin{proof}
By the Strong Noetherian Property~\cite[Thm.\,3.2.2]{CLip}, we may write
\[ f(x,y) = \sum_{\abs{i} \leq d} x^i f_i(y) U_i(x,y) \]
such that for each multi-index $i$, we have $\norm{U_i} = 1$ and $U_i(x,y) \in B_{m,\delta}^{\times}$, for some $\delta > 0$.
As the top slice of each $U_i(x,y)$ is a unit in $\RR[[x,y]]$, it follows that $U_i(0,0)^{\circ} \neq 0$. 
Also note that $U_i(0,\varepsilon)^{\circ} = U_i(0,0)^{\circ}$. 
Together, this implies $\norm{U_i(0,\varepsilon)} = 1$,.

We now prove that $\norm{f_i(\varepsilon)} \leq 1$ for all $i$ with $\abs{i} \leq d$, by induction for the coordinate-wise partial order on $\NN^n$. 
For $i =0$, the product $f_0(\varepsilon) U_0(0,\varepsilon)$ is the constant term of $f(x,\varepsilon)$. 
This implies that $\norm{f_0(\varepsilon) U_0(0,\varepsilon)} \leq 1$, which guarantees $\norm{f_0(\varepsilon)} \leq 1$, since $\norm{U_0(0,\varepsilon)} = 1$.
Now consider the case $\abs{i} > 0$. Write $J$ for the set of all $j \in \NN^n \setminus \{0\}$ such that $i - j \in \NN^n$. 
For all $j \in J$, we may write
\[ U_j(x,y) = \sum_{k \in \NN^n} x^k V_{j,k}(y), \]
where for all $k$ we have $\norm{V_{j,k}} \leq 1$.
Then the coefficient of $x^{i}$ in $f_0(x,\varepsilon)$ is
\[ f_i(\varepsilon) U_i(0,\varepsilon) + \sum_{j \in J} f_j(\varepsilon) V_{j,i-j}(\varepsilon). \]
By induction, the sum over $J$ belongs to $\RR((G))^{\circ}$. Again using that $\norm{U_i(0,\varepsilon)} = 1$, it follows that also $\norm{f_i(\varepsilon)} \leq 1$.

For each $i$ with $\abs{i} \leq d$, write $f_i(\varepsilon) = c_i + \gamma_i$, with $c_i \in \RR \subset B_0$ and $\gamma_i \in \Kinf$.
Let $z$ be a new tuple of variables, indexed by $I = \{ i\in  \NN^n \mid \abs{i} \leq d\}$. Then
\[ g(x,y,v) = \sum_{\abs{i} \leq d} x^i (c_i + v_i) U_i(x,y)  \]
satisfies the conditions of (\ref{eq:g-conditions}), with $z = (y,v)$ and $h = (\operatorname{id},(f_i(y) -c_i)_i )$.
\end{proof}
\begin{proof}[Proof of Lemma~\ref{lem:extended_Weierstrass}]
We verify the conditions of Definition \ref{def:real.weierstrass.system}. 
Conditions (\ref{it:ax_radii}),(\ref{it:ax_coeff}),(\ref{it:ax_rescaling}) and (\ref{it:ax_permutation}) follow in a straightforward way from the corresponding conditions for $\cB$.
For condition (\ref{it:ax_norm}), take any nonzero $f(x,\varepsilon) \in B_{m,\alpha}[A]$ and write $f(x,\varepsilon) = \sum_{\mu} f_{\mu}(\varepsilon) x^{\mu}$.
Since $f(x,\varepsilon)$ belongs to the full Weierstrass system $\cA((G))$ it has a well-defined Gauss norm $\norm{f(x,\varepsilon)}$, which is the maximum among all $\norm{f_{\mu}(\varepsilon)}$. 
Now let $\mu_0$ be any $\mu$ for which this maximum is attained. Then $\norm{ f_{\mu_0}(\varepsilon)^{-1} f(x,\varepsilon)  } = 1$, as required.
	
	
%
%
	

To verify that Weierstrass division holds, let $f(x,\varepsilon) \in B_{m,\alpha}$ be regular in $x_m$ and let $k(x,\gamma) \in B_{m,\alpha}[A]$ be arbitrary.
Lemma \ref{lem:change_of_var} gives a $g(x,z)$ which is regular in $x_m$. 
Then Weierstrass division in $\cB$ yields
\[ k(x,u) = g(x,z) Q(x,z,u)  + R(x,z,u) \]
in some $B_{m + n + \ell + k,\delta}$, where $\delta \in \RR$ does not depend on $k(x,u)$.
Now plug in the tuples of infinitesimals $z = h(\varepsilon)$ and $u = \gamma$.

Finally, since condition (\ref{it:ax_strong})(3) allows parameters, it follows that $\cB[A]$ is a strong real Weierstrass system whenever $\cB$ is.
\end{proof}

\begin{proof}[Proof of Proposition~\ref{prop:1.terms}]
By Theorem \ref{thm:embedding}, we may view $K$ as an $\cL_\cB$-substructure of some $\RR((G))$, equipped with the natural $\cB$-analytic structure, where $\cB\subset \cA((G))$.
Let $\cB[A] \subset \cA((G))$ be constructed as in Definition \ref{def:expanded-Weierstrass}, using the inclusions $A \subset K \subset \RR((G)) $.

If $t(x)$ is an $\cL_{\cB[A]}$-term, then by construction it is also an $\cL_{\cB}(A)$-term. Conversely, since we have symbols for the elements of $A$ in $\cL_{\cB[A]}$, any $\cL_{\cB[A]}$-term is also naturally a $\cL_{\cB}(A)$-term.
\end{proof}
\subsection{Rings of analytic functions}
To analyze 1-terms with parameters in $\cL_{\cB}$, we will use the notion of rings of analytic functions from \cite{CLRr}.
First, we recall some notation for intervals and annuli.
Given $c,r \in K$, $r > 0$, write $I(c,r)$ for the closed interval with center $c$ and radius $r$:
\[ I(c,r) = \{ x \in K \mid \abs{x - c} \leq r \}  .\]
For $c,\delta,\varepsilon \in K$, we denote by $D(c,\delta,\varepsilon)$ the annulus with center $c$, inner radius $\delta$ and outer radius $\varepsilon$:
\[ D(c,\delta,\varepsilon) = \{ x \in K \mid \delta < \abs{x - c} < \varepsilon \}. \]
\begin{defn} \label{def:an_ring}
	Let $K$ be a field with $\cB$-analytic structure, where $\cB$ is a strong real Weierstrass system.
	For any interval $I = I(c,r)$ and any annulus $D = D(c,\delta,\varepsilon)$, define the rings of analytic functions on $I$ and $D$ as follows:
	\begin{align*}
		\cO_I(\cB) &\coloneqq \left\{ f\left(\frac{x - c}{r}\right) \middle| f \in B_{1,\alpha} \text{ for some } \alpha \in \RR_{>1}  \right\},  \\
		\cO_D(\cB) &\coloneq \left\{ g\left(\frac{\delta}{x - c}\right) + h \left(\frac{x - c}{\varepsilon}\right) 
		\middle| g,h \in B_{1,\alpha} \text{ for some } \alpha \in \RR_{>1} \text{ and } g(0) = 0 \right\}.
	\end{align*}
	We will also write $\cO_I$, $\cO_D$ instead of $\cO_I(\cB)$, $\cO_D(\cB)$ when the Weierstrass system $\cB$ is clear from the context.
\end{defn}
\begin{remark}
	The requirement that $\cB$ is strong, guarantees that $\cO_D(\cB)$ is, in fact, a ring.
	For more details, see \cite{CLip} and \cite[Remark~3.1]{CLRr}.
\end{remark}
\begin{remark}
	For each $I = I(c,r)$ and $D = D(c,\delta,\varepsilon)$, the rings $\cO_I$ and $\cO_D$ are normed $\RR$-vector spaces, for the respective norms
	\begin{align*}
		&\left \lVert {f\left(\frac{x - c}{r}\right)} \right \rVert = \norm{f}, \\
		&\left \lVert {h\left(\frac{x - c}{\varepsilon}\right) + g\left(\frac{\delta}{x - c}\right)} \right \rVert = \max\{\norm{h},\norm{g}\},
	\end{align*}
	where on the right-hand side $\norm{\cdot}$ stands for the Gauss norm in $B_{1,\alpha}$, for an appropriate $\alpha > 1$.
	Moreover, the norm on $\cO_I$ is multiplicative and the norm on $\cO_D$ is submultiplicative, see~\cite[Remark~3.1]{CLRr}.
\end{remark}
Let $v: K^\times\to \Gamma$ be any convex valuation on $K$, which is automatically henselian by~\cite[Prop.\,2.9]{delon-farre}. We will denote this valuation additively. Note that $v$ is a coarsening of the natural valuation $\norm{\cdot}$.
We first prove a series of lemmas, showing how to control $\rv_\lambda$ of these units.
\begin{lem} \label{lem:Lip-cont}
	Let $f \in B_{1,\alpha}$ with $\norm{f} \leq 1$, $\alpha > 1$ and $a,b \in [-1,1]_K$ .
	Then $v(f(a) - f(b)) \geq v(a - b)$.
\end{lem}
\begin{proof}
In the case where $a-b$ is not infinitesimal, we are done by Lemma \ref{le:automatic_continuity}.
When $a-b$ is infinitesimal, we may reduce to the case where both $a,b$ are inifinitesimal, by axiom (\ref{it:ax_rescaling}).
Weierstrass divide $f(x)$ by $(x-y)$ to obtain some $g(x,y) \in B_{2,\delta}^{\circ}$ such that
\[ f(x) - f(y) = (x-y) g(x,y) \]
and conclude by Lemma \ref{le:automatic_continuity} applied to $g$.
\end{proof}

\begin{lem} \label{lem:rv-I-unit}
Let $\lambda\in \Gamma_{\geq 0}$. If $I = I(c,r)$ is an interval and $h(x) \in \cO_I^{\times}$, then $\rv_{\lambda}(h(x))$ only depends on $\rv_{\lambda}(x - c)$.
\end{lem}
\begin{proof}
	By assumption, $f(x) = h(r x + c)$ is a unit in some $B_{1,\alpha}$ with $\alpha > 1$.
	Multiplying by a constant, we may assume that $\norm{f} = \norm{f^{-1}} = 1$.
	We need to show that $v(f(a) - f(b)) > v(f(a)) + \lambda$ whenever $v(a  - b) > v(a) + \lambda$ for $a,b \in [-1,1]_K$.
	Since both $\norm{f(a)} \leq 1$ and $\norm{f^{-1}(a)} \leq 1$, it follows that $\norm{f(a)} = 1$.
	Hence, $v(f(a)) = 0$ for all $a \in [-1,1]_K$ and we conclude by Lemma~\ref{lem:Lip-cont}.
\end{proof}
\begin{lem} \label{lem:rv-strong-unit}
	If $D = D(c,\delta,\varepsilon)$ is an annulus and $U(x) \in \cO_D$ is a strong unit (i.e. $\norm{U(x) -1 } < 1$), then for every $\lambda \in \Gamma_{\geq 0}$ the value of $\rv_{\lambda}(U(x))$ only depends on $\rv_{\lambda}(x-c)$.
\end{lem}
\begin{proof}
	Apply a translation and rescaling to reduce to the case where $D = D(0,\delta,1)$, with $\abs{\delta} < 1$.
	Since $U(x)$ is a strong unit, we can write $U(x) = 1 +  g(\delta/x) + h(x)$, with $g,h \in B_{1,\alpha}$ for some $\alpha > 1$ and $\norm{g},\norm{h} < 1$.
	Now let $a,b$ be distinct elements in $D$ such that $\rv_{\lambda}(a) = \rv_{\lambda}(b)$. 
	We apply Lemma \ref{lem:Lip-cont} to compute that
	\begin{align*}
		v(U(a) - U(b)) & = v( (h(a) - h(b)) + (g(\delta/a) - g(\delta/b) ) 	\\ 
		&	\geq \min\{ v(a - b),  v( b - a) - (v(ab) - v(\delta) ) \}   \\
		&	\geq \min\{ v(a-b), v(a - b) - v(a) \} \\
		& 	> \lambda, 
	\end{align*}
	where we used that $v(a) = v(b) \leq v(\delta)$ and $v(a- b) > v(a) + \lambda$.
	Because $v(U(a)) = 0$, the above computation shows that $\rv_{\lambda}(U(a))  = \rv_{\lambda}(U(b))$ 
\end{proof}

\subsection{Term analysis for real closed fields}\label{sec:term.real.closed} We now investigate the behavior of terms in the case where $K$ is real closed. This essentially follows work by Cluckers--Lipshitz--Robinson~\cite{CLRr}. For this section we assume $\cB$ to be strong and rich, and our field $K$ with $\cB$-analytic structure to be real closed. Recall that we have a convex valuation $v$ on $K$ with value group $\Gamma_K$.

\begin{prop} \label{prop:term_structure}
Let $A,K$ be two real closed fields with $\cB$-analytic structure, such that $A$ is an $\cL_{\cB}$-substructure of $K$, and let $\lambda\in \Gamma_{A, \geq 0}$.
Let $\tau(x)$ be an $\cL_\cB(A)$-term in one variable. 
Then there exist finitely many $a_0 = -\infty < a_1 < \ldots < a_n = \infty$ in $A$ and $\cL_{\cB[A]}$-terms $\tau_i(x)$, and finitely many $c_1,\dots,c_m \in A$, such that on each interval $(a_i, a_{i+1})_{K}$ we have that
\[
\tau(x) = R_i(x)\cdot \tau_i(x),
\]
for some rational function $R_i$ over $A$, and $\rv_{\lambda}(\tau_i(x))$ depends only on the tuple
\[
 (\rv_{\lambda}(x - c_i))_{i = 1}^m.
\]
\end{prop} 

The below proposition is an extension of \cite[Prop.~3.13]{CLRr}. 
The main difference is that we need control on the parameters from $A$. 
This requires functions from $\cB[A]$ to be used in the definition of the $\cO_{X_i}$ rather than only functions from $\cB$.
\begin{prop} \label{prop:terms_belong_to_an_ring}
Let $\cB$ be a strong real Weierstrass system over a divisible group $G$ and let $A$ be a real closed field which is an $\cL_\cB$-substructure of $\RR((G))$.
Let $\tau(x)$ be an $\cL_\cB(A)$-term in one variable. 

There exists a covering of $[-1,1]_{\RR((G))}$ by finitely many intervals and annuli $X_i$ with centers and radii in $A$ such that except for finitely many values of $x\in [-1,1]_{\RR((G))}$, we have for each $i$ that $\tau(x) \in \cO_{X_i}(\cB[A])$.
\end{prop}

In other words, in the above proposition we will have that $\tau(x)$ agrees with an element $\tau_i(x) \in \cO_{X_i}(\cB[A])$ on each $X_i$, except at finitely many points of $X_i$.

\begin{proof}
In \cite[Prop.~3.13]{CLRr} it is shown that this proposition holds when $A = \RR((G))$. 
Hence, there exist $X_i$, $\tau_i$ as in the statement but with $X_i$ defined over $\RR((G))$ and $\tau_i \in \cO_{X_i}(\RR((G)))$.
Now this proposition essentially follows from the fact that $A \preccurlyeq \RR((G))$ as $\cL_{\cB[A]}$-structures. We give some details.

First observe that there exists an $\cL_\cB$-formula in free variables $c,r,a$ expressing that on the interval $I = I(c,r)$ it holds that (except at certain finite number of points)
\begin{equation} \label{eq:term_OI_generalized}
	\tau(x) = f((x-c)/r,a)
\end{equation}
for a certain $f \in B_{1+n,\alpha}$, and some $n \in \NN$, $\alpha > 1$.
Similarly, there is an $\cL_\cB$-formula in free variables $c,\delta,\varepsilon,a,b$ expressing that on the annulus $D = D(c,\delta,\varepsilon)$ it holds that (expect at certain finite number of points)
\begin{equation} \label{eq:term_OD_generalized}
	\tau(x) = g(\delta/(x-c),a) + h((x - c)/\varepsilon,b) 
\end{equation}
for specific $g,h \in B_{1+n,\alpha}$ (some $n \in \NN$, $\alpha > 1$). 

Now recall that we have quantifier elimination in $\cL_\cB$ for the theory of real closed fields with $\cB$-analytic structure, by~\cite[Thm.~3.4.3]{CLip}. 
Hence, $A \preccurlyeq \RR((G))$ and there exists a covering of $[-1,1]_A$ by finitely many intervals and annuli with centers and radii in $A$, such that on each annulus $\tau(x)$ is of the form (\ref{eq:term_OD_generalized}) and on each interval $\tau(x)$ is of the form (\ref{eq:term_OI_generalized}) for $a,b \in A$ (except at a finite number of points).
If all $a,b$ appearing in (\ref{eq:term_OD_generalized}) and (\ref{eq:term_OI_generalized}) can be taken infinitesimal, then the right-hand sides belong to $\cO_I(\cB[A])$ and $\cO_{D}(\cB[A])$ respectively, and we are done.
Suppose this is not the case for some annulus $D$.
Then apply axiom (\ref{it:ax_rescaling}) and compactness of $[-1,1]_{\RR}$ to cut up $D$ into finitely many annuli $D_1,\dots,D_q$ such that $\tau_{|D}$ belongs to $\cO_{D_j}(\cB[A])$ for every $j = 1,\dots,q$, except at finitely many points. The case of an interval is similar.

Let $X_1',\dots,X'_{n'}$ be the resulting covering of $[-1,1]_A$ by intervals and annuli and $\tau_i' \in \cO_{X_i'}(\cB[A])$ the corresponding terms.
Since $A \preccurlyeq \RR((G))$ for $\cL_{\cB[A]}$, it follows that the formulas defining $X_i'$ and $\tau_i'$ then work over $\RR((G))$ as well.
This concludes the proof.
\end{proof}
\begin{proof}[Proof of Proposition \ref{prop:term_structure}]
Theorem \ref{thm:embedding} produces an $\cL_\cB$-embedding $K \hookrightarrow \RR((G))$, where the latter is equipped with the natural $\cB$-analytic structure.
Similarly to Proposition \ref{prop:terms_belong_to_an_ring}, we may use that $A \preccurlyeq K \preccurlyeq \RR((G))$ as $\cL_{\cB[A]}$-structures to reduce to the case where $A = K = \RR((G))$ and $\cB = \cB[A] = \cB[K]$, and where $\cB$ is still strong, but not necessarily rich.
Furthermore, performing the change of variables $x \mapsto 1/x$, it suffices to consider only $x \in [-1,1]_K$, so that Proposition \ref{prop:terms_belong_to_an_ring} applies.

This yields a covering of $[-1,1]_K$ by finitely many intervals and annuli $X_i$ such that on each of them $\tau(x)$ agrees with an element of $\cO_{X_i}(\cB[A])$, except at finitely many points.
Fix any interval or annulus $X$ of this cover of $[-1,1]_{K}$.
If $X$ is an interval, then similarly to \cite[Cor.~3.5]{CLRr} we may find finitely many intervals $\{I_i\}_{i = 1}^n$ covering $X$, polynomials $P_i(x) \in K[x]$ and units $U_i(x) \in \cO_{I_i}^{\times}$ such that $\tau_{|I_i}(x) = P_i(x) U_i(x)$ except at finitely many points, for all $i = 1,\dots,n$. Then use Lemma~\ref{lem:rv-I-unit} to control $\rv_{\lambda}(U_i(x))$.

If $X = D(c,\delta,\varepsilon)$ is an annulus, then apply the same arguments as in \cite[Lemma~3.7]{CLRr}. 
We may split of any number of intervals, since we already know how to deal with these.
So let $D = D(c,\delta,\varepsilon)$ be one of the finitely many resulting annuli.  
The arguments of \cite[Lemma 3.7]{CLRr} show that
\[ \tau_{|D}(x) = R(x) h(x) U(x), \]
where $R(x)$ is a rational function $h(x)$ is a unit in $\cO_I(c,\varepsilon)$ (and not just $\cO_{D}$) and $U(x)$ is a strong unit.
Note that the above equality between \emph{terms} then also holds on the intervals $(c-\varepsilon,c-\delta)$ and $(c+\delta,c+\varepsilon)$ seperately.
Now use Lemmas~\ref{lem:rv-I-unit} and~\ref{lem:rv-strong-unit} to find finitely many $c_1,\dots,c_n \in K$ such that each occurring $\rv_{\lambda}(h(x))$ and $\rv_{\lambda}(U(x))$ is completely determined by the finite tuple $(\rv_{\lambda}(x - c_i))_{ i = 1}^n$.
Note that the finitely many points where $\tau(x)$ does not belong to any $\cO_X$ are harmless, as we can simply add these points to the set of $a_i$'s.
\end{proof}

\subsection{Pulling back from the real closure}\label{sec:pulling.back}

In Proposition~\ref{prop:term_structure} we have shown how to prepare an $\cL_{\cB}(A)$-term in one variable when the field is real closed. We now show how to pull back our preparing set from the real closure to the ground field. More concretely, the goal is to prove the following.

\begin{lem}\label{lem:pull_back_from_rcl}
Let $A\subset K$ be a subfield, and let $D\subset A^\rcl$ be a finite set. Then there exists finite set $C\subset K$ which is $\cL_{\mathrm{val}}(A)$-definable such that for every $\lambda\in \Gamma_{K,\geq 0}$, every ball $B\subset K$ $\lambda$-next to $C$ is contained in a ball $\lambda$-next to $D$.
\end{lem}

\begin{proof}
We prove the following claim via induction:
\begin{enumerate}
\item[] Let $p\in A[x]$ be a non-constant polynomial with roots $\alpha_1, \ldots, \alpha_n$ in $A^\rcl$. Then there exists a finite $\cL_\val(A)$-definable set $C\subset K$ such that if $x,y\in K$ are in the same ball $1$-next to $C$, then $\rv(x-\alpha_i) = \rv(y-\alpha_i)$ for every $i$.
\end{enumerate}

We use induction on $n = \deg (p)$. If $n=1$ it is clear, since then we can simply take $C = \{\alpha_1\}$. So we assume that $n > 1$ and that the statement is proven for smaller $n$. By induction, let $C$ be a finite set for the claim applied to $p'$, and also such that $\rv(p(x))$ is constant on balls $1$-next to $C$, which is possible by~\cite[Prop.~3.6]{Flen}. This $C$ is $\cL_\val(A)$-definable, and we claim that it suffices. So let $x,y$ be in the same ball $1$-next to $C$, and assume that the claim is false. Then there exists some root $\alpha$ of $p$ such that $\rv(x-\alpha)\neq \rv(x-\alpha)$. Let $B\subset K^\rcl$ be the closed valuative ball around $x$ of radius $v(x-y)$, then $\alpha$ is an element of $B$. If $B$ would contain another root $\alpha'\in K^\rcl$ of $p$, then by Rolle's theorem it would also contain a root of $p'$, since $K^\rcl$ is real closed and $B$ is convex. But this contradicts the construction of $C$, and so the only root of $p$ in $B$ is $\alpha$. Hence if $\alpha'$ is another root of $p$, then $\rv(x-\alpha') = \rv(y-\alpha')$. But if we now compute
\[
\rv(p(x)) = \rv(x-\alpha) \prod_{\alpha'\neq \alpha} \rv(x-\alpha') \neq \rv(y-\alpha') \prod_{\alpha'\neq \alpha} \rv(y-\alpha) = \rv(p(y)),
\]
which contradicts the construction of $C$. This concludes the proof of the claim.

We now turn to proving the lemma, for which we apply the claim to all minimal polynomials of elements of $D$ to obtain a finite $\cL_{\val}(A)$-definable set $C$ in $K$. Now if $B\subset K$ is a ball $1$-next to $C$ then by construction it is contained in a ball $B'$ which is $1$-next to $D$. Let $B_\lambda\subset B$ be a ball which is $\lambda$-next to $C$. Then the radius of $B_\lambda$ is $\lambda + \radop (B)$, and so $B_\lambda$ is contained in a ball $B_\lambda'\subset B'$ in $K^\rcl$ of radius $\lambda + \radop (B')$. But then $B'_\lambda$ is $\lambda$-next to $D$, as desired.
\end{proof}

\begin{remark}
This lemma in fact also holds when $D$ is a subset of $A^\alg$, the algebraic closure of $A$, with the same proof. In that case one should replace the use of Rolle's theorem by the fact that if $p\in K^{\alg}[x]$ has two roots in the valuation ring, then its derivative also has a root in the valuation ring.
\end{remark}

We can finally prove that we may prepare terms in a good way.

\begin{prop}\label{prop:term_analysis}
Let $K$ be an almost real closed field with real analytic $\cB$-structure, where $\cB$ is strong and rich. Let $A\subset K$ and let $\tau(x)$ be an $\cL_{\cB}(A)$-term in one variable. Then there exists a finite $\cL_{\val, \cB}(A)$-definable set $C\subset K$ such that for each $\lambda\in \Gamma_{\geq 0}$, the value of $\rv_{\lambda}(\tau(x))$ depends only on the tuple $(\rv_{\lambda}(x-c))_{c\in C}$. 
\end{prop}

\begin{proof}
Without loss of generality, we may assume that $\dcl_{\cL_{\cB}}(A) = A$, in particular $A$ is a subfield of $K$. We apply Proposition~\ref{prop:term_structure} to $A^{\rcl}$ in $K^{\rcl}$ to the term $\tau(x)$ to find a finite set $D\subset A^{\rcl}$ such that for $x\in K^{\rcl}$ the value of $\rv_{\lambda}(\tau(x))$ depends only on the tuple $(\rv_{\lambda}(x-d))_{d\in D}$. We now conclude by applying Lemma~\ref{lem:pull_back_from_rcl} to the set $D$.
\end{proof}

\section{Relative quantifier elimination}\label{sec:QE}

In this section we prove two relative quantifier elimination results. 


\subsection{Eliminating to $\RV_\lambda$}\label{sec:QE.RV}

The proof of h-minimality is based on relative quantifier elimination results in a larger many-sorted language.
First we apply the strategy of Denef--van den Dries \cite{DvdD} to reduce to the case where the variables over which we quantify occur only polynomially. Then we apply the relative quantifier elimination of result of Flenner \cite{Flen} for the pure valued field structure.
To go from relative quantifier elimination to $\omega$-h-minimality, we follow the method from~\cite[Sec.\,6]{CHR}.

For our first relative quantifier elimination result we work in a two-sorted setting with a valued field sort and an $\RV$-sort. We denote by $\cL_{\RV}$ the language consisting of the valued field language $\cL_{\val} = \{0,1,+,\cdot,\cO\}$ on the valued fields sort, the language $\{0,1,\cdot, \oplus\}$ on the $\RV$-sort, and a symbol $\rv$ for a map from the valued field to $\RV$. Any valued field may be interpreted as an $\cL_{\RV}$-structure where all symbols have their usual meaning, and $\oplus$ is a ternary relation on $\RV$ interpreted as
\[
\oplus(\alpha, \beta, \gamma) \Leftrightarrow \exists x,y,z\in K: \alpha=\rv(x) \wedge \beta =\rv(y)\wedge \gamma=\rv(z)\wedge x+y=z.
\]
We call this the \emph{partial addition}. We recall a first relative quantifier elimination result due to Flenner (see also \cite{Basarab} for an earlier result in a related language).

\begin{theorem}[{{\cite[Prop.\,4.3]{Flen}}}]\label{thm:flenner}
Let $K$ be a henselian valued field of equicharacteristic zero. Then $\Th_{\cL_{\RV}}(K)$ eliminates valued field quantifiers. Moreover, this still holds when the language is expanded by adding structure purely on $\RV$.
\end{theorem}

However, we need quantifier elimination after adding symbols for the elements of a Weierstrass system. Moreover, to prove $\omega$-h-minimality, we need quantifier elimination to $\RV_\lambda$. So let $\cB$ be a real Weierstrass system which is strong, let $K$ be an almost real closed field equipped with $\cB$-analytic structure, and take $\lambda\in \Gamma_{K,\geq 0}$. Let $\cL_{\RV_\lambda,\cB}$ be the language with a valued field sort and an $\RV_\lambda$-sort, where we take the language $\cL_{\val, \cB}$ on the valued field sort, the map $\rv_\lambda$ from the valued field to $\RV_\lambda$, and the full induced structure on $\RV_\lambda$. In other words, we add a predicate for each $\emptyset$-definable subset of $\RV_\lambda^n$, for each $n$. 

\begin{prop}\label{prop:QE.RV}
With notation as above, $\Th_{\cL_{\RV_\lambda,\cB}}(K)$ eliminates valued field quantifiers.
\end{prop}

\begin{proof}
First note that any valued field quantifier-free $\cL_{\RV_\lambda,\cB}$-formula is equivalent to one of the form $\phi(\rv_\lambda t_1(x, y), \ldots, \rv_\lambda t_r(x, y))$ where $\phi$ is a formula in the $\RV_\lambda$-sort with $\RV_\lambda$-parameters, and the $t_i$ are $\cL_{\cB, \RV_\lambda}$-terms with $x$ and $y = (y_1, \ldots, y_n)$ valued field variables. So it suffices to eliminate the field quantifier from a formula of the form
\[
\exists x \, \phi(\rv_\lambda t_1(x, y), \ldots, \rv_\lambda t_r(x, y)).
\]
By cutting up the domain and using the map $a \mapsto 1/a$, we can assume that all variables run over $[-1,1]_K$. We can introduce new variables so that the formula $\phi(\rv_\lambda t_1(x, y), \ldots, \rv_\lambda t_r(x, y))$ is equivalent to a formula of the form
\[
\exists z_1, \ldots, z_m\, \theta (\rv_\lambda f_1(z, y), \ldots, \rv_\lambda f_s(z, y)),
\]
where $z = (z_1, \ldots, z_m)$ are valued field variables, the $f_i$ are either in $B_{m+n, \alpha}$ or a polynomial over $K$, and no field division occurs. Let $(z_0, y_0)\in [-1,1]^{m+n}$. Using the Strong Noetherian property from~\cite[Theorem 3.2.2]{CLip}, there exist an integer $d$ and a real number $\beta(z_0, y_0)>0$, unit elements $u_{ij}\in B^{\circ}_{m+n,\beta}$, elements $h_{ij}\in B_{n,\beta}$ and  a subset $J_i$ of $\{0,...,d\}^m$ such that for every $i = 1, \ldots, s$,
\[
f_i(z, y)=\sum_{j\in J_i} h_{ij}(y-y_0)(z-z_0)^j u_{ij}(z-z_0,y-y_0)
\]
if $|z-z_0|, |y-y_0|<\beta(z_0,y_0)$. Since the real cube $[-1,1]_{\RR}^{m+n}$ is compact, we only need to eliminate the quantifiers from
\[
\exists z  \,\, |z-z_0| < \beta(z_0, y_0)\wedge |y-y_0| < \beta(z_0, y_0)\wedge \theta (\rv_\lambda f_1(z, y), \ldots, \rv_\lambda f_s(z, y)),
\]
for finitely many points $(z_0, y_0)\in [-1, 1]_{\RR}^{m+n}$. Without loss of generality, we can assume that $(z_0, y_0) = (0,0)$. Now, by a coordinate transformation and Weierstrass division similar to~\cite[Proof of Theorem 2.14]{CLRr} and \cite[Lemma 4.13 and Proof 4.14]{DvdD}, we can shrink $\beta$ if needed and write $f_i(z,y)$ as 
\[
h_{ij}(y)v_i\left(y,\frac{h_{i \ell}(y)}{h_{j \ell}(y)}-c, z\right) g_i\left(y,\frac{h_{i \ell}(y)}{h_{j \ell}(y)}-c, z\right),
\]
for some fixed $j$, $c\in [-1,1]_{\RR}$, where $v_i$ is a unit and $g_i$ polynomial in the last variable $z_m$. Now, if $v\in B_{n', \alpha}^\circ$ is a unit, then by construction of our language, there exists a definable map $\overline{v}$ on $\RV_\lambda^{n'}$ such that $\rv_\lambda(v(x)) = \overline{v}(\rv_\lambda(x))$. Hence $\rv_\lambda(f_i(z,y))$ depends only on $\rv_\lambda(z,y), \rv_\lambda(h_{i \ell}(y)/h_{j \ell}(y)-c)$ and $\rv_\lambda(g_i( y,h_{i \ell}(y)/h_{j \ell}(y)-c, z))$. In other words, the variable $z_m$ only occurs polynomially in our formula (with some parameters depending on $y$). If $\lambda = 1$, we can then conclude from Theorem~\ref{thm:flenner} to eliminate the quantifier over $z_m$ and iterate this procedure. If $\lambda\neq 1$, then the result follows from~\cite[Prop.\,6.2.3]{CHR}.
\end{proof}

\subsection{Eliminating to the residue field and value group}\label{sec:QE.RF.VG}

In this section we prove a variant of the relative quantifier elimination result from the previous section, down to the residue field and the value group. We will choose a more specific language, rather than taking the full induced structure on $\RV$, which will give us an understanding of the induced structure on the residue field and the value group.

Let $\cL_{\ac}$ denote the three-sorted language with sorts $\VF, \RF$ and $\VG$ with $\cL_{\val}$ on $\VG$, the ring language on $\RF$, and $\cL_{\oag}$ on $\VG$. We add the valuation $\norm{\cdot}$ from $\VF$ to $\VG$, and an angular component map $\ac: \VF\to \RF$, which we recall is a multiplicative morphism $\ac: \VF^\times \to \RF^\times$ extending the residue map on $\cO^\times$. In this setting, Pas has proven relative quantifier elimination (see also \cite{Wei76} for a similar result in a related language).

\begin{theorem}[{{\cite[Thm.\,4.1]{Pas}}}]\label{thm:Pas.QE}
Let $K$ be a henselian valued field of equicharacteristic zero with an angular component map $\ac$. Then $\Th_{\cL_{\ac}}(K)$ eliminates valued field quantifiers.
\end{theorem}

We extend this result by adding analytic structure from a real Weierstrass system. So let $\cB$ be a real Weierstrass system over $\Omega$ which is rich. From $\cB$ we construct another Weierstrass system $\cB'$ which we will use for the language on the residue field. The Weierstrass system $\cB'$ will have value group $0$, and $B'_{n, \alpha}\subset \RR[[\xi_1, \ldots, \xi_n]]$ is generated as an algebra over $\RR[\xi_1, \ldots, \xi_n]$ by all top slices of units of $\cB_{n, \alpha}$. Note that this is indeed a real Weierstrass system, as each element of $\cB'$ is a top slice of some element of $\cB$. As an example, if $\cB$ is the full Weierstrass system $\cA((\Omega))$, then $\cB'$ is the full Weierstrass system $\cA((0))$. In other words, $\cB'$ essentially consists of restrictions of analytic functions  to compact subsets of $\RR^n$.

We consider the following three-sorted language $\cL_{\ac, \cB}$ with sorts $\VF$, $\RF$ and $\VG$ for the valued field, the residue field, and the value group respectively. We equip $\VF$ with the language $\cL_{\val, \cB} = \cL_{\val}\cup \cL_{\cB}$, $\VG$ with the language of ordered abelian groups $\cL_{\oag}$, and $\RF$ with language $\cL_{\cB'}$. As maps between the sorts, we add the valuation $\norm{\cdot}$ from $\VF$ to $\VG$, and an angular component map $\ac$ from $\VF$ to $\RF$. Recall that $B_0$ sits naturally inside $\VF$, and since $\cB$ is rich $\RR[\Omega]$ is a subset of $B_0$. If $K$ is an $\cL_{\ac, \cB}$-structure, we say that $\ac$ and $\cB$ are \emph{compatible} if for each element $\omega\in \Omega$, $\ac(t^\omega) = 1$. 

\begin{lem}
Let $\cB$ be a real Weierstrass system over $\Omega$ which is rich, and let $K$ be an almost real closed field with $\cB$-analytic structure. Then $K$ equipped with the natural valuation has $\cL_{\ac, \cB}$-structure such that $\ac$ and $\cB$ are compatible.
\end{lem}

\begin{proof}
We interpret all symbols of $\cL_{\ac, \cB}$, except for $\ac$, in the usual way. Note that the residue field of $K$ is simply $\RR$, so that the language on the residue field makes sense. 

We need to construct a compatible angular component map. For this, recall that $\Omega$ is a subgroup of the valuation group $G$ of $K$, and hence by Lemma~\ref{lem:extend_section} the partial section $\Omega\subset G\to K: \omega\mapsto t^\omega$ extends to a section $s: G\to K$. Now define $\ac(x) = \res(s(\norm{x})^{-1}x)$ to conclude.
\end{proof}

From now on, whenever we consider an almost real closed field $K$ as an $\cL_{\ac, \cB}$-structure, we always assume that $\ac$ and $\cB$ are compatible. 

\begin{prop}\label{prop:QE.RF.VG}
Let $\cB$ be a rich real Weierstrass sytem, and let $K$ be an almost real closed field. Then $\Th_{\cL_{\ac, \cB}}(K)$ eliminates valued field quantifiers. In particular, the sorts $\RF$ and $\VG$ are orthogonal.
\end{prop}

\begin{proof}
The proof is largely the same as the proof of Proposition~\ref{prop:QE.RV}, but we give some details. By syntactical considerations, it suffices to eliminate the valued field quantifier from a formula of the form
\[
\exists\, x \phi(\ac (t_1(x,y), \ldots, \ac(t_r(x,y))\wedge \psi(\norm{t_1(x,y)}, \ldots, \norm{t_r(x,y)}),
\]
where $\phi$ is a formula (with parameters) over $\RF$, $\psi$ is a formula (with parameters) over $\VG$, and $t_1, \ldots, t_r$ are $\cL_{\cB}$-terms. Now follow the proof of Proposition~\ref{prop:QE.RV} exactly, until one has written $f_i(z,y)$ on $[-\beta, \beta]_K$ as
\[
h_{ij}(y)v_i\left(y,\frac{h_{i \ell}(y)}{h_{j \ell}(y)}-c, z\right) g_i\left(y,\frac{h_{i \ell}(y)}{h_{j \ell}(y)}-c, z\right),
\]
where $v_i$ is a unit in $\cB$ and $g_i$ depends polynomially on the last variable $z_m$. Our formula depends only on $\ac (f_j(z,y))$ and $\norm{f_j(y,z)}$. Now if $v$ is a unit, its top slice $v_{0}$ is a symbol in the language on the residue field and we have that $\ac(v(x)) = v_0(\ac(x))$. Since also $\norm{v} = 1$, we conclude that our formula depends only polynomially on $z_m$. Hence by Theorem~\ref{thm:Pas.QE} we may eliminate the valued field quantifier, and iterate the procedure.

The last fact about orthogonality follows directly from the relative quantifier elimination and our choice of language.
\end{proof}

\begin{remark}
The compatibility assumption is included to simplify the treatment. If one drops the assumption that $\ac$ and $\cB$ are compatible, we have to put slightly more structure on the residue field. Namely, we have to name constants from $B_0$ by including symbols for the values of $\ac(t^\omega)$, for each $\omega\in \Omega$. Note that this uses the richness assumption. In this larger language, one has relative quantifier elimination again. See also Remark~\ref{rem:ac.cB.not.compatible}.
\end{remark}

\begin{remark}\label{rem:any.convex.valuation}
For an arbitrary convex valuation, one can still obtain relative quantifier elimination by adding symbols for all elements of $\cB$ to the language on the residue field. However, it is then not so clear what the language on the residue field actually is. We expect that in general, the residue field will be another almost real closed field equipped with $\cB''$-analytic structure, for some related Weierstrass system $\cB''$ over a quotient $\Omega''$ of $\Omega$. In fact, this is what happens above for the natural valuation. Indeed, the new Weierstrass system $\cB'$ constructed above has value group $\Omega' = 0$, which is related to the fact that the residue field $\RR$ has value group $0$.
\end{remark}

\section{Proofs of our main results}

\subsection{Hensel minimality}\label{sec:h.min.proof}

With all the results from the previous sections, we can quickly deduce that almost real closed fields equipped with analytic structure are $\omega$-h-minimal. 

\begin{proof}[Proof of Theorem~\ref{thm:real.h.minimal}]
Recall that $\cB$ is a real Weierstrass system which is strong and rich, that $K$ is an almost real closed field with real analytic $\cB$-structure, and that $\cO_K$ is a convex valuation ring of $K$. We wish to prove that the theory of $K$ in $\cL_{\val,\cB}$ is $\omega$-h-minimal.

Let $K'$ be elementarily equivalent to $K$, $A\subset K'$ and $X\subset K'$ be $(A\cup \RV_{K', \lambda})$-definable, for some $\lambda$ in the value group of $K'$. Let $\phi(x)$ be an $(A\cup \RV_{K', \lambda})$-formula defining $X$. 

By Lemma~\ref{prop:QE.RV} we may assume that $\phi(x)$ contains no valued field quantifiers, and hence by inspection it is enough to $\lambda$-prepare finitely many functions of the form $x\mapsto \rv_{\lambda}(\tau(x))$, where $\tau$ is an $A$-term in one variable. Now apply Proposition~\ref{prop:term_analysis} to conclude.
\end{proof}

We briefly discuss some consequences of $\omega$-h-minimality about tameness of definable objects. Let $K$ be an almost real closed field equipped with $\cB$-analytic structure, for some real Weierstrass system which is rich and strong. We fix some convex valuation $v: K^\times \to \Gamma$ on $K$, then we know that the resulting theory is $\omega$-h-minimal in the language $\cL_{\val, \cB}$. 

We begin with some results which follow directly from~\cite{CHR}.

\begin{enumerate}
\item (Jacobian property~\cite[Lem.\,2.8.5]{CHR}) If $f: K\to K$ is a definable function, then there exists a finite set $C\subset K$ such that for each ball $B$ disjoint from $C$ there exists a $\lambda\in \Gamma$ such that for all $x,y\in B$
\[
v(f(x) - f(y)) = v(x-y) + \lambda.
\]
Moreover, we can take $C$ to be definable over the same parameters as $f$.
\item (Differentiation~\cite[Cor.\,3.1.6]{CHR}) Continuing with the notation from the Jacobian property, there exists a finite set $C\subset K$ such that on each ball $B$ disjoint from $C$, $f$ is $C^1$, $v(f'(x))$ is constant for $x\in B$, and for all $x,y\in B$
\[
v(f(x)- f(y)) = v(x-y) + v(f'(x)).
\]
Again, we may take $C$ to be definable over the same parameters as $f$.
\item (Cell decomposition~\cite[Thm.\,5.2.4]{CHR}) There exists a good notion of cells such that every definable set $X\subset K^n$ may be written as a finite union of definable cells. Using cells, one also obtains a good dimension theory of definable sets~\cite[Prop.\,5.3.4]{CHR}.
\item (Resplendency~\cite[Thm.\,4.1.19]{CHR}) If $\cL'\supset \cL_{\val, \cB}$ is any expansion of the language by predicates on Cartesian powers of $\RV$, then $\Th_{\cL'}(K)$ is also $\omega$-h-minimal. For example, we may add an angular component map $\ac: K\to k$ by adding a splitting $\Gamma\to \RV$, and the resulting theory will still be $\omega$-h-minimal.
\end{enumerate}

Since $\Th_{\cL_{\val, \cB}}(K)$ is $\omega$-h-minimal, we obtain also slow growth of rational points on definable transcendental curves~\cite{CNSV}.

\begin{example}\label{ex:counting.dim}
We use the notion of rational points and counting dimension from~\cite{CNSV}. Consider $K = \RR((t))$ for the natural valuation and equipped with analytic structure from the full Weierstrass system $\cA((\ZZ))$. Let $X$ be the graph of the exponential function $\exp: [-1,1]_K\to K$, which is a definable transcendental curve. For $s\geq 1$ an integer, let $X_s$ denote the rational points of height at most $s$, i.e.\ those points $(x,y)\in X$ for which $x,y\in \RR[t]$ are polynomials of degree at most $s-1$. Then $X_1$ is infinite, as it contains all points of the form $(x,\exp(x))$ for $x\in \RR$. However, for every $s\geq 1$ we have that $X_s = X_1$, and so the counting dimension of $X$ is bounded by $(1,1,1)$. 

In general, if $X\subset K^n$ is a definable transcendental curve, then~\cite[Thm.\,2.2.1]{CNSV} shows that for every $\varepsilon>0$, the counting dimension of $X$ is bounded by $(N_\varepsilon, 1, \lceil \varepsilon s\rceil)$, where $N_\varepsilon$ is a constant depending on $\varepsilon$.
\end{example} 

Additionally, $1$-h-minimal fields automatically satisfy many other nice topological properties, by \cite{Now22,Now24,Now24closed}.
Note that the statements \cite{Now24} require that the residue field and value group are orthogonal, and that the structure on the latter is the pure ordered abelian group structure.
By Theorem \ref{thm:induced.structure.intro}, our ordered fields with real analytic structure do indeed satisfy this additional condition.

\subsection{The induced structure on the residue field and value group}

Let $\cB$ be a real Weierstrass system which is strong and rich, and recall the construction of the Weierstrass system $\cB'$ and the resulting three-sorted language $\cL_{\ac, \cB}$ from Section~\ref{sec:QE.RF.VG}. 

\begin{theorem}\label{thm:induced.structure}
Let $\cB$ be a real Weierstrass system which is strong and rich. Let $K$ be an almost real closed field equipped with $\cB$-analytic structure, considered as a valued field for the natural valuation, and equip it with $\cL_{\ac, \cB}$-structure. Assume that $\ac$ and $\cB$ are compatible, then
\begin{enumerate}
	\item the theory $\Th_{\cL_{\ac, \cB}}(K)$ is $\omega$-h-minimal,
	\item the definable subsets in the residue field are those definable in $\cL_{\cB'}$,
	\item the definable subsets in the value group are those definable in the language of ordered abelian groups,
	\item the residue field and value group are stably embedded and orthogonal.
\end{enumerate}
\end{theorem}

\begin{proof}
The first part follows from resplendency~\cite[Thm.\,4.1.19]{CHR}, since adding an $\ac$ map to $\cL_{\val, \cB}$ is an $\RV$-expansion. The other parts follow directly from the relative quantifier elimination result from Proposition~\ref{prop:QE.RF.VG}.
\end{proof}

In particular, if $\cB$ is the full Weierstrass system, note that the residue field becomes the structure $\RR_\an$.

\begin{remark}
The above result is complementary to the main Theorem by Cubides-Kovacsics and Haskell in \cite{CubidesHaskell}. They investigate real closed fields with separated and overconvergent analytic structure (instead of real analytic structure), and prove quantifier elimination in a one-sorted language. From this they deduce that such structures are weakly o-minimal. 
	
In our setting, an appropriate version of the latter also holds: if $K$ is a real closed field with $\cB$-analytic structure for some real Weierstrass system $\cB$, then $T_{\cL_{\cB}}(K)$ is weakly o-minimal. Indeed, as $T_{\cL_{\cB}}(K)$ is the expansion of an o-minimal structure by a convex predicate, it is weakly o-minimal by~\cite[{\S}~4]{BP98} (see also~\cite{Bai01}). 
\end{remark}
 
\subsection{An Ax--Kochen--Ersov theorem}

We can now prove our Ax--Kochen--Ersov theorem, continuing with the language and set-up from above. 

\begin{theorem}\label{thm:AKE}
Let $\cB$ be a real Weierstrass system which is strong and rich. Let $K$ and $K'$ be almost real closed fields equipped with $\cB$-analytic structure, considered as valued fields for the natural valuation, and equip them with $\cL_{\ac, \cB}$-structure. Let $G$ and $G'$ be the value groups of $K$ and $K'$ and assume that $\ac$ and $\cB$ are compatible. Then
\[
K \equiv_{\cL_{\ac, \cB}} K' \text{ if and only if } G \equiv_{\cL_{\oag}} G'.
\]
\end{theorem}

\begin{proof}
By Theorem~\ref{thm:induced.structure}, the induced structure on the residue field of $K$ and $K'$ is exactly the same. Also, the induced structure on the value group is exactly the structure of an ordered abelian group. Hence the result follows from the relative quantifier elimination from Proposition~\ref{prop:QE.RF.VG}.
\end{proof}

\begin{remark}\label{rem:ac.cB.not.compatible}
Note that in this theorem, we tacitly assume that $\ac$ and $\cB$ are compatible. Without this assumption, the residue fields carries slightly more structure by adding symbols for $\ac(t^\omega)$ for every $\omega\in \Omega$. In particular, this result remains true without the compatibility assumption, but one has to add the requirement that the residue fields are elementarily equivalent with this extra structure of naming constants.
\end{remark}

\begin{remark}
We expect that one can prove a similar AKE principle when considering arbitrary convex valuations on $K$ and $K'$, again via relative quantifier elimination. However, it is not entirely clear what the structure on the residue field will be. As explained in Remark~\ref{rem:any.convex.valuation}, we expect that the residue field is itself another almost real closed field with analytic structure (from some different Weierstrass system).
\end{remark}

\bibliographystyle{amsplain}
\bibliography{anbib}
\end{document}